\documentclass[11pt,oneside,reqno]{amsart}
\usepackage{amssymb}
\usepackage{algpseudocode}
\usepackage{algorithm}
\usepackage{verbatim}
\usepackage{graphicx}
\usepackage{placeins}
\usepackage{listings}
\usepackage{xcolor}
\usepackage{subcaption}
\usepackage[noadjust]{cite}
\RequirePackage{dsfont} 
 \scrollmode
\allowdisplaybreaks
\usepackage{graphicx}
\usepackage{epstopdf}
\usepackage{hyperref} 
\usepackage{mathtools}

\usepackage{amsmath}
\usepackage{tikz}


\definecolor{codegreen}{rgb}{0,0.6,0}
\definecolor{codegray}{rgb}{0.5,0.5,0.5}
\definecolor{codepurple}{rgb}{0.58,0,0.82}
\definecolor{backcolour}{rgb}{0.95,0.95,0.92}

\lstdefinestyle{list_style}{
  backgroundcolor=\color{backcolour}, commentstyle=\color{codegreen},
  keywordstyle=\color{magenta},
  numberstyle=\tiny\color{codegray},
  stringstyle=\color{codepurple},
  basicstyle=\ttfamily\footnotesize,
  breakatwhitespace=false,         
  breaklines=true,                 
  captionpos=b,                    
  keepspaces=true,                 
  numbers=left,                    
  numbersep=5pt,                  
  showspaces=false,                
  showstringspaces=false,
  showtabs=false,                  
  tabsize=2
}

\newcommand{\xdasharrow}[2][->]{
\tikz[baseline=-\the\dimexpr\fontdimen22\textfont2\relax]{
\node[anchor=south,font=\scriptsize, inner ysep=1.5pt,outer xsep=2.2pt](x){#2};
\draw[shorten <=3.4pt,shorten >=3.4pt,dashed,#1](x.south west)--(x.south east);
}
}

\newcommand{\DEBUG}{}

\ifdefined\DEBUG%

  \def\rem#1{{\marginpar{\raggedright\scriptsize #1}}}
  \newcommand{\pmr}[1]{\rem{\color{blue}{$\bullet$ #1}}}
  \newcommand{\ppr}[1]{\rem{\color{red}{$\bullet$ #1}}}
 \else%

  \newcommand{\ppr}[1]{}
  \newcommand{\pmr}[1]{}
 \fi

\newcommand{\R}{{\mathbb R}}

\newcommand{\E}{{\mathbb E}}

\def\rho{\varrho_1}

\def\rd{\,{\mathrm d}}

\theoremstyle{plain}
\newtheorem{theorem}{Theorem}
\newtheorem{lemma}{Lemma}
\newtheorem{fact}{Fact}

\newtheorem{proposition}{Proposition}

\theoremstyle{definition}

\begin{document}

\title
[Approximation of SDEs under inexact information]
{On the randomized Euler algorithm under inexact information}

\author[M. Baranek]{Marcin Baranek}
\address{AGH University of Krakow,
	Faculty of Applied Mathematics,
	Al. A.~Mickiewicza 30, 30-059 Krak\'ow, Poland}
\email{mbaranek@agh.edu.pl}

\author[A. Ka\l u\.za]{Andrzej Ka\l u\.za}
\address{AGH University of Krakow,
	Faculty of Applied Mathematics,
	Al. A.~Mickiewicza 30, 30-059 Krak\'ow, Poland}
\email{akaluza@agh.edu.pl}

\author[P. M. Morkisz]{Pawe{\l } M. Morkisz}
\address{NVIDIA Corp. and AGH University of Krakow,
	Faculty of Applied Mathematics,
	Al. A.~Mickiewicza 30, 30-059 Krak\'ow, Poland}
\email{morkiszp@agh.edu.pl}

\author[P. Przyby{\l}owicz]{Pawe{\l} Przyby{\l}owicz}
\address{AGH University of Krakow,
Faculty of Applied Mathematics,
Al. A.~Mickiewicza 30, 30-059 Krak\'ow, Poland}
\email{pprzybyl@agh.edu.pl, corresponding author}

\author[M. Sobieraj]{Micha{\l} Sobieraj}
\address{AGH University of Krakow,
	Faculty of Applied Mathematics,
	Al. A.~Mickiewicza 30, 30-059 Krak\'ow, Poland}
\email{sobieraj@agh.edu.pl}
	
\begin{abstract}
This paper focuses on analyzing the error of the randomized Euler algorithm when only noisy information about the coefficients of the underlying stochastic differential equation (SDE) and the driving Wiener process is available. Two classes of disturbed Wiener process are considered, and the dependence of the algorithm's error on the regularity of the disturbing functions is investigated. The paper also presents results from numerical experiments to support the theoretical findings. 
\newline
\newline
\textbf{Key words:} stochastic differential equations, randomized Euler algorithm, inexact information, Wiener process, lower bounds, optimality
\newline
\newline
\textbf{MSC 2010:} 65C30, 68Q25
\end{abstract}
\maketitle

\section{Introduction}
We investigate the strong approximation of solutions of the following SDEs
\begin{equation}
\label{main_equation}
	\left\{ \begin{array}{ll}
	\displaystyle{
	\rd X(t) = a(t,X(t))\rd t + b(t,X(t)) \rd W(t), \ t\in [0,T]},\\
	X(0)=\eta, 
	\end{array} \right.
\end{equation}
where $T >0$, $W$ is an $m$-dimensional Wiener process, 
 and $\eta\in\mathbb{R}^d$. Our analysis is performed under the assumption that only {\it standard noisy information} about $(a,b,W)$ is available. This means that we have access to $a,b,W$ only through its {\it inexact values}  at finite number of discretization points.

Our interest lies in approximating the values of $X(T)$ using the inexact information about the coefficients $(a,b)$ and the driving Wiener process $W$. We consider algorithms that are based on values of $a,b$, and $W$ corrupted by noise. This noise can arise from measurement errors, rounding procedures, etc.  The inspiration of considering such inexact information comes from various sources, such as numerically solving SDEs  on GPUs and understanding of impact of low precision in computations (when switching from {\bf double} to {\bf float} and {\bf half}, see \cite{AKAPhD}), as well as modeling real-world phenomena that are described by SDEs such as energy demand/production forecasting (where exact information is rarely available).

The study of inexact information has been explored in the literature for various problems, including  function approximation and integration (\cite{KaPl90}, \cite{milvic}, \cite{MoPl16}, \cite{Pla14}), approximate solving of ODEs (\cite{KaPr16}) and PDEs (\cite{Wer96}, \cite{Wer97}), see also the related monograph \cite{Pla96}. In the context of stochastic integration and approximation of solution of stochastic differential equations inexact information about the integrands or coefficients of the underlying SDEs has been considered in \cite{AKPMPP}, \cite{PMPP17}, \cite{PMPP19}. However, it is important to note that in \cite{PMPP17} and \cite{PMPP19} the information considered about the process $W$ was exact. We also refer to the article \cite{GSM2022} where noisy information induced by the approximation of normally distributed random variables is considered. However, the computational setting (devoted for weak approximation of SDEs) is different that the one considered in this paper (established in the context of strong approximation of the solution $X$).

In this paper, we mainly extend the proof technique known from \cite{PMPP17} and \cite{PMPP19}. Namely, we  cover the case when also the information about the Wiener process $W$ is inexact. This assumption leads to a significant change in the proof technique. It allows  us  to investigate the error behavior for the randomized Euler scheme under inexact information about the tuple $(a,b,W)$ with precision parameters $\delta_1,\delta_2,\delta_3\in [0,1]$ for $a,b,W$, respectively. (See also, for example, \cite{JenNeuen}, \cite{PMPP14}, \cite{KRWU_0}, \cite{KRWU} where other randomized algorithms for approximation of solutions of ODEs and SDEs have been defined and investigated under exact information.) Roughly speaking, we show that the $L^r(\Omega)$-error of the randomized Euler scheme, that use $O(n)$ noisy evaluations of $(a,b,W)$, is $O(n^{-\min\{\varrho,1/2\}}+\delta_1+\delta_2+\delta_3)$, provided that the corrupting functions for $W$ are sufficiently regular (see Theorem \ref{theorem_error_bounds} (i). In the case of less regular corrupting functions for $W$ (assuming only H\"older continuity) the error might increase due to the presence of informational noise (see Theorem \ref{theorem_error_bounds} (ii)).

The main contributions of this paper are as follows:
\begin{itemize}
    \item Upper error bounds on the randomized Euler algorithm in two classes of corrupting functions for the Wiener process $W$ (Theorem \ref{theorem_error_bounds}),
    \item Lower error bounds and optimality of the randomized Euler algorithm (Theorem \ref{thm_lower_bounds}),
    \item Results of numerical experiments that confirm our theoretical findings (Section 5).
\end{itemize}
The structure of the paper is as follows. Section 2 provides basic notions and definitions, along with a description of the computation model used when dealing with inexact information for drift and diffusion coefficients, as well as for the driving Wiener process. In Section 3, we analyze the upper bounds for the error of the randomized Euler algorithm. Lower bounds and some optimality results are stated in Section 4. Section 5 contains the results of numerical experiments conducted to validate our theoretical findings. Finally, the Appendix provides auxiliary results used in the paper. 
\section{Preliminaries}
We denote by $\mathbb{N}=\{1,2,\ldots\}$. Let  $W = \{W(t)\}_{t\geq 0}$ be a standard $m$-dimensional Wiener process defined on a complete probability space  $(\Omega,\Sigma, \mathbb{P})$. By $\{\Sigma_t\}_{t\geq 0}$ we denote a filtration, satisfying the usual conditions, such that $W$ is a Wiener process with respect to $\{\Sigma_t\}_{t\geq 0}$. We set $\Sigma_{\infty}=\sigma\Bigl(\bigcup_{t\geq 0}\Sigma_t\Bigr)$.
 We denote by $\|\cdot\|$  the Frobenius norm in $\mathbb{R}^m$ or $\mathbb{R}^{d\times m}$ respectively, where we treat a column vector in $\mathbb{R}^m$ as a matrix of size $m\times 1$. For $x\in\mathbb{R}^m$ and $\alpha\in\mathbb{R}$, by $x\cdot\alpha$ or $\alpha\cdot x$ we mean a componentwise scalar-by-vector multiplication and for a matrix $y\in\mathbb{R}^{d\times m}$  by $y\cdot x$ we mean a standard matrix-by-vector multiplication. For a sufficiently smooth function $f:[0,T]\times\mathbb{R}^m\to\mathbb{R}$ we denote by $\partial f(t,y)/\partial y$ its gradient, while by $\partial^2 f(t,y)/\partial y^2$ its Hessian matrix  of size $m\times m$. Moreover, for a  smooth function $f:[0,T]\times\mathbb{R}^m\to\mathbb{R}^m$ we denote by $\partial f(t,y)/\partial y$ its Jacobi matrix of size $m\times m$, computed also with respect to the space variable $y$. For $r\in [2,+\infty)$ by the $L^r(\Omega)$-norm, either for a random vector or a random matrix, we mean
\begin{equation}
 \|Y\|_r := \left( \E\|Y\|^r\right)^{1/r}\qquad\mbox{for}\qquad Y:\Omega \to \mathbb{R}^m \mbox{ or } Y:\Omega \to\mathbb{R}^{d\times m}.
\end{equation}
We also make us of the  following second order differential operator
\begin{equation}
    \mathcal{L}=\frac{\partial}{\partial t}+\frac{1}{2}\sum\limits_{k=1}^m\frac{\partial^2}{\partial y_k^2}.
\end{equation}
We now define classes of drift and diffusion coefficients.
Let $T>0$, $K>0$, $\varrho\in (0,1]$. A function $a:[0,T]\times \mathbb{R}^d \to \mathbb{R}^d$ belongs to $\mathcal{A}_{K}$ if
\begin{itemize}
\item the mapping $a:[0,T]\times\mathbb{R}^d\to\mathbb{R}^d$ is Borel measurable, 
\item for all $ t\in[0,T]$
\begin{equation}
\|a(t,0)\| \leq K,
\end{equation}
\item for all $t\in[0,T]$, $x,y \in \mathbb{R}^d$
\begin{equation} 
	\|a(t,x) - a(t,y)\| \leq K \|x - y\|.
\end{equation}
\end{itemize}
Note that if $a\in\mathcal{A}_{K}$ then for all $(t,y)\in [0,T]\times\mathbb{R}^d$ we have
\begin{equation}
	\|a(t,y)\| \leq K(1+\|y\|).
\end{equation}
A mapping $b:[0,T]\times\mathbb{R}^d\to\mathbb{R}^{d\times m}$ belongs to $\mathcal{B}^{\varrho}_{K}$ if
\begin{itemize}
\item $b$ is bounded in the origin $(0,0)$,
	\begin{equation}
	\|b(0,0)\| \leq K,
	\end{equation}
\item for all $t\in[0,T]$, $x,y \in \mathbb{R}^d$
	\begin{equation}
		\label{lip_b1}
			\|b(t,x) - b(t,y) \| \leq K  \| x-y\|,
	\end{equation}
\item for all $t,s\in[0,T]$, $x \in \mathbb{R}^d$
	\begin{equation}
		\label{lip_b2}
			\|b(t,x) - b(s,x) \| \leq K (1+\|x\|)\cdot |t-s|^{\varrho}.
	\end{equation}	
\end{itemize}
The above conditions imply that for all $(t,y)\in [0,T]\times\mathbb{R}^d$
\begin{equation}
	\label{ling_g_b}
 \|b(t,y)\| \leq\bar K (1+\| y\|),
\end{equation}
where $\bar K=K\cdot\max\{1,T^{\varrho}\}$. We also consider the following class of initial values
\begin{equation}
	\mathcal{J}_K=\{\eta\in\mathbb{R}^d \ | \ \|\eta\|\leq K\}.
\end{equation}
The class of all admissible tuples $(a,b,\eta)$ is defined as
\begin{equation}
    \mathcal{F}(\varrho,K)=\mathcal{A}_K\times\mathcal{B}^{\varrho}_K\times\mathcal{J}_K.
\end{equation}
Let 
\begin{equation}
	\delta_1,\delta_2, \delta_3,\delta_4 \in [0,1]
\end{equation}	
We refer to $\delta_1$, $\delta_2$, $\delta_3$, and $\delta_4$ as to \textit{precision parameters}. 
We now describe what we mean by corrupted values and information about $a, b, W$.

Let us set
\begin{eqnarray*}
	\mathcal{K}^s&=&\{p:[0,T]\times\mathbb{R}^d\to \mathbb{R}^{d\times s} \ | \ p(\cdot,\cdot)-\hbox{Borel measurable}, \\
    && \|p(t,y)\|\leq 1+\|y\| \ \hbox{for all}  \  t\in [0,T], y\in\mathbb{R}^d\},
\end{eqnarray*}
for $s\in\{1,m\}$. The classes $\mathcal{K}^1$, $\mathcal{K}^m$ are nonempty and contain constant functions.  Let
\begin{equation}
	V_c(\gamma)=\{\tilde c \ | \ \exists_{p_c\in\mathcal{K}^s}: \tilde c=c+\gamma\cdot p_c\},
\end{equation}
where $c\in\{a,b\}$, $(\gamma,s)=(\delta_1,1)$ if $c=a$ and $(\gamma,s)=(\delta_2,m)$ if $c=b$. By $\tilde a$ and $\tilde b$ we mean any functions $\tilde a\in V_a(\delta_1)$ and $\tilde b\in V_b(\delta_2)$, respectively. We have that $\{a\}=V_a(0)\subset V_a(\delta_1)\subset V_a(\delta'_1)$ for $0\leq \delta_1\leq \delta'_1\leq 1$ and $\{b\}=V_b(0)\subset V_b(\delta_2)\subset V_b(\delta'_2)$ for $0\leq \delta_2\leq \delta'_2\leq 1$.

In order to introduce perturbed information about the Wiener process $W$, we introduce the following  classes of corrupting functions for $W$
\begin{eqnarray}
     &&\mathcal{K}_0=\Bigl\{p:[0,T]\times\mathbb{R}^m\to\mathbb{R}^m \ | \ p^j\in C^{1,2}([0,T]\times\mathbb{R}^m;\mathbb{R}), \ |p^j(0,0)|\leq 1,\notag\\
     &&\quad\quad\quad\quad\max\Bigl\{\Bigl|\frac{\partial p^j}{\partial t}(t,y)\Bigl|,\Bigl\|\frac{\partial p^j}{\partial y}(t,y)\Bigl\|,\Bigl\|\frac{\partial^2 p^j}{\partial y^2}(t,y)\Bigl\|\Bigr\}\leq 1\notag\\
     &&\quad\quad\quad\quad\hbox{for all} \ t\in [0,T], y\in\mathbb{R}^m, j=1,2,\ldots,m\Bigr\},
\end{eqnarray}
and
\begin{eqnarray}
    &&\mathcal{K}_{\alpha,\beta} =
        \Bigl\{
            p:[0,T]\times\mathbb{R}^m\to\mathbb{R}^m \ 
            | \ \Vert p(t,x)-p(s, y)\Vert \leq \vert t-s\vert^\alpha
            + \Vert x-y\Vert^\beta,\notag\\
             &&\quad\quad\quad\quad\hbox{for all} \ t,s\in [0,T], x,y\in\mathbb{R}^m\Bigr\}.
\end{eqnarray}
We consider the following classes of disturbed Wiener processes
\begin{equation}
	\mathcal{W}_0(\delta_3)=\{\tilde W \ | \ \exists_{p\in\mathcal{ K}_0}:\forall_{(t,\omega)\in [0,T]\times\Omega} \ \tilde W(t,\omega) = W(t,\omega)+\delta_3\cdot p(t,W(t,\omega))\},
\end{equation}
and
\begin{equation}
	\mathcal{W}_{\alpha,\beta}(\delta_3)=\{\tilde W \ | \ \exists_{p\in\mathcal{K}
	_{\alpha,\beta}}:\forall_{(t,\omega)\in [0,T]\times\Omega} \ \tilde W(t,\omega) = W(t,\omega)+\delta_3\cdot p(t,W(t,\omega))\}.
\end{equation}
We have that $\{W\}=\mathcal{W}_0(0)\subset\mathcal{W}_0(\delta_3)\subset\mathcal{W}_0(\delta'_3)$ for $0\leq \delta_3\leq\delta'_3\leq 1$, and similarly for $\mathcal{W}_{\alpha,\beta}$. As in \cite{AKPMPP}  the classes defined above  allow us to model the impact of regularity of  noise on the error bound.

We assume that the algorithm is based on discrete noisy information about $(a,b,W)$ and exact information about $\eta$. Hence, a vector of noisy information has the following form
\begin{eqnarray}
	\mathcal{N}(\tilde a, \tilde b, \tilde W, \eta)&=&\Bigl[\tilde a (\xi_0,y_0),\tilde a(\xi_1,y_1),\ldots,\tilde a(\xi_{i_1-1},y_{i_1-1}),\notag\\
		 && \tilde b(t_0,z_0), \tilde b(t_1,z_1),\ldots, \tilde b(t_{i_1-1},z_{i_1-1}), \notag\\
		&& \tilde W(u_0), \tilde W(u_1),\ldots,\tilde   W(u_{i_2-1}),\eta],
\end{eqnarray}
where  $i_1,i_2\in\mathbb{N}$ and $(\xi_0,\xi_1,\ldots,\xi_{i_1-1})$ is a random vector on $(\Omega,\Sigma,\mathbb{P})$ which takes values in $[0,T]^{i_1}$. We assume that the $\sigma$-fields $\sigma(\xi_0,\xi_1,\ldots,\xi_{i_1-1})$ and $\Sigma_{\infty}$ are independent. Moreover,   $t_0,t_1,\ldots,t_{i_1-1}\in [0,T]$ and $u_0,u_1,\ldots,u_{i_2-1}\in [0,T]$ are fixed time points. The evaluation points $y_j$, $z_j$ for the spatial variables $y,z$ of $a(\cdot,y)$ and $b(\cdot,z)$ can be computed in adaptive way with respect to $(a,b,\eta)$ and $W$. Formally, it means that  there exist Borel measurable mappings $\psi_0:\mathbb{R}^{i_2\times m}\times\mathbb{R}^d\to\mathbb{R}^{2d}$, $\psi_j:\mathbb{R}^{d\times j}\times\mathbb{R}^{d\times m\times j}\times\mathbb{R}^{i_2\times m}\times\mathbb{R}^d\to\mathbb{R}^{2d}$, $j=1,2,\ldots,i_1-1$, such that the successive points $y_j,z_j$ are given as follows
\begin{equation}
	(y_0,z_0)=\psi_0\Bigl(\tilde W(u_0),\tilde W(u_1),\ldots,\tilde  W(u_{i_2-1}),\eta\Bigr), 
\end{equation}
and for $j=1,2,\ldots, i_1-1$
\begin{eqnarray}
	(y_j,z_j)&=&\psi_j\Bigl(\tilde a(\xi_0,y_0), \tilde a(\xi_1,y_1),\ldots, \tilde a(\xi_{j-1},y_{j-1}),\notag\\
	&&\quad\quad \tilde b(t_0,z_0), \tilde b(t_1,z_1),\ldots, \tilde b(t_{j-1},z_{j-1}),\notag\\
	&&\quad\quad \tilde W(u_0),\tilde W(u_1),\ldots, \tilde W(u_{i_2-1}),\eta\Bigr). 
\end{eqnarray}
The total number of noisy evaluations of $(a,b,W)$ is $l = 2 i_1 + i_2$.

The algorithm $\mathcal{A}$ that uses the noisy information $\mathcal{N}(\tilde a, \tilde b, \tilde W, \eta)$ and computes approximation of $X(T)$ is defined as 
\begin{equation}
\label{alg_def}
    \mathcal{A}(\tilde a,\tilde b, \tilde W,\eta)=\varphi\Bigl(\mathcal{N}(\tilde a,\tilde b, \tilde W,\eta)\Bigr),
\end{equation}
for some Borel measurable function $\varphi:\mathbb{R}^{i_1\times d}\times\mathbb{R}^{i_1\times d\times m}\times\mathbb{R}^{i_2\times m}\times\mathbb{R}^d\to\mathbb{R}^d$. For a fixed $n\in\mathbb{N}$ by $\Phi_n$ we denote a class of all algorithms \eqref{alg_def} for which the total number of evaluations $l$ is at most $n$.

Let $r\in [2,+\infty)$. The $L^r(\Omega)$-error of  $\mathcal{A}\in\Phi_n$ for the fixed tuple $(a,b,\eta)\in\mathcal{G}$ is given by 
\begin{eqnarray}
    &&e^{(r)}(\mathcal{A},a,b,\eta,\mathcal{W},\delta_1,\delta_2,\delta_3)\notag\\
    &&\quad\quad=\sup\limits_{(\tilde a,\tilde b, \tilde W)\in V_a(\delta_1)\times V_b(\delta_2)\times\mathcal{W}(\delta_3)}\|X(a,b,W,\eta)(T)-\mathcal{A}(\tilde a,\tilde b, \tilde W,\eta)\|_r,
\end{eqnarray}
where $\mathcal{W}\in\{\mathcal{W}_0,\mathcal{W}_{\alpha,\beta}\}$ and $\mathcal{G}$ is a subclass of $\mathcal{F}(\varrho,K)$. The worst case error of $\mathcal{A}$ in $\mathcal{G}$ is
\begin{equation}
    e^{(r)}(\mathcal{A},\mathcal{G},\mathcal{W},\delta_1,\delta_2,\delta_3)=\sup\limits_{(a,b,\eta)\in\mathcal{G}} e^{(r)}(\mathcal{A},a,b,\eta,\mathcal{W},\delta_1,\delta_2,\delta_3).
\end{equation}
Finally, we look for  (essentially) sharp bounds for the $n$th minimal error, defined as 
\begin{equation}
\label{def_nminerr}
e^{(r)}_n(\mathcal{G},\mathcal{W},\delta_1,\delta_2,\delta_3)=\inf\limits_{\mathcal{A}\in\Phi_n}e^{(r)}(\mathcal{A},\mathcal{G},\mathcal{W},\delta_1,\delta_2,\delta_3).
\end{equation}
In \eqref{def_nminerr} we define the minimal possible error among all algorithms of the form \eqref{alg_def} that use at most $n$ noisy evaluations of $a,b$ and $W$ . Our aim is to find possibly sharp bounds on the $n$th minimal error, i.e., lower and upper bounds which match up to constants. We are also interested in defining an algorithm for which the infimum in $e^{(r)}_n(\mathcal{G},\mathcal{W},\delta_1,\delta_2,\delta_3)$ is asymptotically attained.  We call such an algorithm the optimal one. 

Unless otherwise stated, all constants appearing in this paper (including those in the 'O', '$\Omega$', and '$\Theta$' notation) will only depend on the parameters of the class $\mathcal{F}(\varrho,K)$, $\alpha,\beta$ and $r$. Furthermore, the same symbol may be used in order to denote  different constants.
\section{Error of the Euler scheme under inexact information}
We investigate the error of the randomized Euler scheme in the case of inexact information  about $a$, $b$, and the driving Wiener process $W$.

Fix $n\in\mathbb{N}$, $t_i=iT/n$ for $i=0,1,\ldots,n$. Let $\{\xi_i\}_{i=0}^{n-1}$ be independent random variables on $(\Omega,\Sigma,\mathbb{P})$, such
that the $\sigma$-fields $\sigma(\xi_0, \xi_1,\ldots, \xi_{n-1})$ and $\Sigma_{\infty}$ are independent, with $\xi_i$ being uniformly distributed on $[t_i,t_{i+1}]$. Let us fix $(a,b,W)\in \mathcal{F}(\varrho,K)$ and take any $(\tilde a,\tilde b, \tilde W)\in V_a(\delta_1)\times V_b(\delta_2)\times\mathcal{W}(\delta_3)$, where $\mathcal{W}\in\{\mathcal{W}_0,\mathcal{W}_{\alpha,\beta}\}$. The randomized Euler scheme under inexact information is defined by taking
\begin{equation}
    \label{EU_NO1}
    		\bar X^{RE}_n(0)=\eta,
\end{equation}
and
\begin{equation}
	\label{EU_NO2}
			\bar X^{RE}_n(t_{i+1})   = \bar X^{RE}_n(t_i) + \tilde a(\xi_i, \bar X^{RE}_n(t_i)) \cdot\frac{T}{n} +  \tilde b(t_i, \bar X^{RE}_n(t_i))  \cdot \Delta \tilde{W}_i,
\end{equation}
for $i=0,1, \ldots, n-1$, where $\Delta \tilde W_i = \tilde W(t_{i+1}) - \tilde W(t_i)$. The randomized Euler algorithm $\mathcal{\bar A}^{RE}_{n}$ is defined as
\begin{equation}
    \mathcal{\bar A}^{RE}_{n}(\tilde a,\tilde b,\tilde W,\eta):=\bar X^{RE}_n(T).
\end{equation}
The informational cost of the randomized Euler algorithm is $O(n)$ noisy evaluations of $a,b,W$. By  $X^{RE}_{n}$ we denote the randomized Euler algorithm $\bar X^{RE}_{n}$ under the case when information is exact, i.e., when $\delta_1=\delta_2=\delta_3=0$. 

Let $\mathcal{G}^n=\sigma(\xi_0,\xi_1,\ldots,\xi_{n-1})$ and $\tilde\Sigma_t^n=\sigma\Bigl(\Sigma_t\cup\mathcal{G}^n\Bigr)$, $t\geq 0$. Since the $\sigma$-fields $\Sigma_{\infty}$ and $\mathcal{G}^n$ are independent, the process $W$ is still the  $m$-dimensional Wiener process on $(\Omega,\Sigma,\mathbb{P})$ with respect to $\{\tilde\Sigma_t^n\}_{t\geq 0}$.
\begin{proposition}  \label{prop1} Let $r\in [2,+\infty)$, $\varrho\in (0,1]$. 
\begin{itemize}
    \item [(i)]  There exists  $C\in (0,+\infty)$, depending only on the parameters of the class $\mathcal{F}(\varrho,K)$ and $r$, such that for all $n\in\mathbb{N}$, $\delta_1,\delta_2,\delta_3\in [0,1]$, $(a,b,\eta)\in\mathcal{F}(\varrho,K)$, $(\tilde a,\tilde b, \tilde W)\in V_a(\delta_1)\times V_b(\delta_2)\times\mathcal{W}_0(\delta_3)$ it holds
\begin{equation}
    \Bigl\|\max\limits_{0\leq i\leq n}\|X^{RE}_{n}(t_i)-\bar X^{RE}_{n}(t_i)\|\Bigl\|_r\leq C(\delta_1+\delta_2+\delta_3).
\end{equation}
    \item [(ii)] Let $\alpha,\beta\in (0,1]$. There exists  $C\in (0,+\infty)$, depending only on the parameters of the class $\mathcal{F}(\varrho,K)$, $\alpha,\beta$, and $r$, such that for all $n\in\mathbb{N}$, $\delta_1,\delta_2,\delta_3\in [0,1]$, $(a,b,\eta)\in\mathcal{F}(\varrho,K)$, $(\tilde a,\tilde b, \tilde W)\in V_a(\delta_1)\times V_b(\delta_2)\times\mathcal{W}_{\alpha,\beta}(\delta_3)$ it holds
\begin{eqnarray}
    &&\Bigl\|\max\limits_{0\leq i\leq n}\|X^{RE}_{n}(t_i)-\bar X^{RE}_{n}(t_i)\|\Bigl\|_r\notag\\
    &&\leq C(\delta_1+\delta_2+\delta_3\cdot n^{1-\gamma})\cdot (1+\delta_3 n^{1-\gamma})\cdot e^{C(1+(\delta_3 n^{1-\gamma})^r)},
\end{eqnarray}
where $\gamma=\min\{\alpha,\beta/2\}$.
\end{itemize}
\end{proposition}
\begin{proof} Firstly, we prove (i). For $\tilde W\in\mathcal{W}_0(\delta_3)$ we have that
\begin{equation}
    \tilde W(t)=W(t)+\delta_3\cdot Z(t),
\end{equation}
with $Z(t)=p_W(t,W(t))$ and $p_W\in\mathcal{K}_0$. Then, by the It\^o formula we get that
\begin{equation}
	Z(t)=Z(0)+M(t)+V(t), \quad t\in [0,T],
\end{equation}
where $M(t)=[M^1(t),M^2(t),\ldots, M^
m(t)]$, $V(t)=[V^1(t),V^2(t),\ldots, V^m(t)]$ and
\begin{eqnarray}
	V^j(t)&=&\int\limits_0^t\mathcal{L}p^j_W(z,W(z))\rd z,\\
	M^j(t)&=&\sum\limits_{i=1}^m\int\limits_0^t\frac{\partial p^j_W}{\partial y_i}(z,W(z))\rd W^i(z),
\end{eqnarray}
for $j=1,2,\ldots,m$. We stress that $\{V(t)\}_{t\in [0,T]}$ is a continuous process of bounded variation that is adapted to $\{\tilde\Sigma_t^n\}_{t\geq 0}$.
Moreover, since $(M(t),\tilde\Sigma_t^n)_{t\in [0,T]}$ is a martingale, $Z$ is still a continuous semimartingale  with respect to the extended filtration  $\{\tilde\Sigma_t^n\}_{t\geq 0}$. In the sequel we will consider stochastic  integrals, with respect to the semimartingales $W$ and $Z$, of processes that are adapted to the filtration $\{\tilde\Sigma_t^n\}_{t\geq 0}$.

From  \eqref{EU_NO1}, \eqref{EU_NO2}  for $i=0,1,\ldots,n$ we can write that
\begin{eqnarray}
    &&\bar X^{RE}_n(t_i)=\eta+\frac{T}{n}\sum\limits_{j=0}^{i-1}\tilde a(\xi_j,\bar X^{RE}_n(t_j))+\ \sum\limits_{j=0}^{i-1}\tilde b(t_j,\bar X^{RE}_n(t_j))\cdot\Delta W_j\notag\\
    &&\quad\quad\quad\quad\quad\quad+\delta_3\sum\limits_{j=0}^{i-1}\tilde b(t_j,\bar X^{RE}_n(t_j))\cdot\Delta Z_j,
\end{eqnarray}
and
\begin{eqnarray}
    &&X^{RE}_n(t_i)=\eta+\frac{T}{n}\sum\limits_{j=0}^{i-1} a(\xi_j,X^{RE}_n(t_j))+\ \sum\limits_{j=0}^{i-1} b(t_j, X^{RE}_n(t_j))\cdot\Delta W_j.
\end{eqnarray}
Therefore
\begin{eqnarray}
     &&\bar e_i :=  X^{RE}_n(t_i)-\bar X^{RE}_n(t_i)=\sum\limits_{j=0}^{i-1}\int\limits_{t_j}^{t_{j+1}}\Bigl(a(\xi_j, X^{RE}_n(t_j))-\tilde a(\xi_j,\bar X^{RE}_n(t_j))\Bigr)\rd s\notag\\
     &&\quad\quad\quad\quad\quad\quad\quad\quad\quad\quad\quad\quad+\ \sum\limits_{j=0}^{i-1}\int\limits_{t_j}^{t_{j+1}}\Bigl(b(t_j,X^{RE}_n(t_j))-\tilde b(t_j,\bar X^{RE}_n(t_j))\Bigr)\rd W(s)\notag\\
     &&\quad\quad\quad\quad\quad\quad\quad\quad\quad\quad\quad\quad+\ (-\delta_3)\sum\limits_{j=0}^{i-1}\int\limits_{t_j}^{t_{j+1}}\tilde b(t_j,\bar X^{RE}_n(t_j))\rd Z(s)\notag\\
     &&\quad\quad\quad\quad\quad\quad\quad\quad\quad\quad\quad=\ \sum\limits_{j=0}^{i-1}\Bigl(A_{j}+B_{j}+C_{1,j}+C_{2,j}+C_{3,j}\Bigr),
\end{eqnarray}
where
\begin{eqnarray}
     &&A_{j}=\int\limits_{t_j}^{t_{j+1}}\Bigl(a(\xi_j,X^{RE}_n(t_j))-a(\xi_j,\bar X^{RE}_n(t_j))\Bigr)\rd s\notag\\
     &&B_{j}=\int\limits_{t_j}^{t_{j+1}}\Bigl(b(t_j,X^{RE}_n(t_j))-b(t_j,\bar X^{RE}_n(t_j))\Bigr)\rd W(s)\notag\\
     &&C_{1,j}=(-\delta_1)\int\limits_{t_j}^{t_{j+1}}p_a(\xi_j,\bar X^{RE}_n(t_j))\rd s\notag\\
     &&C_{2,j}=(-\delta_2)\int\limits_{t_j}^{t_{j+1}}p_b(t_j,\bar X^{RE}_n(t_j))\rd W(s),
\end{eqnarray}
\begin{eqnarray}
     &&C_{3,j}=(-\delta_3)\int\limits_{t_j}^{t_{j+1}}\tilde b(t_j,\bar X^{RE}_n(t_j))\rd Z(s)\notag\\
     &&=(-\delta_3)\int\limits_{t_j}^{t_{j+1}}\tilde b(t_j,\bar X^{RE}_n(t_j))\rd M(s)+(-\delta_3)\int\limits_{t_j}^{t_{j+1}}\tilde b(t_j,\bar X^{RE}_n(t_j))\rd V(s)\notag\\
     &&=(-\delta_3)\int\limits_{t_j}^{t_{j+1}}\tilde b(t_j,\bar X^{RE}_n(t_j))\frac{\partial p_W}{\partial y}(s,W(s))\rd W(s)\notag\\
     &&+(-\delta_3)\int\limits_{t_j}^{t_{j+1}}\tilde b(t_j,\bar X^{RE}_n(t_j))\mathcal{L}p_W(s,W(s))\rd s.
\end{eqnarray}
Then for all $i=0,1,\ldots,n$
\begin{eqnarray}
     &&\|\bar e_i\|\leq\sum\limits_{j=0}^{i-1}\|A_j\|+\Bigl\|\sum\limits_{j=0}^{i-1}B_j\Bigl\|+\sum\limits_{j=0}^{i-1}\|C_{1,j}\|+\Bigl\|\sum\limits_{j=0}^{i-1}C_{2,j}\Bigl\|+\Bigl\|\sum\limits_{j=0}^{i-1}C_{3,j}\Bigl\|,
\end{eqnarray}
where
\begin{equation}
    \Bigl\|\sum\limits_{j=0}^{i-1}C_{3,j}\Bigl\|\leq\delta_3\cdot\Bigl(C^1_{3,i}+C^2_{3,i}\Bigr),
\end{equation}
with
\begin{eqnarray}
     &&C^1_{3,i}= \Bigl\|\sum\limits_{j=0}^{i-1} \int\limits_{t_j}^{t_{j+1}}\tilde b
     (t_j,\bar X^{RE}_n(t_j))\frac{\partial p_W}{\partial y}(s,W(s))\rd W(s)\Bigl\|,\\
     &&C^2_{3,i}=\sum\limits_{j=0}^{i-1}\Bigl\|\int\limits_{t_j}^{t_{j+1}}\tilde b
     (t_j,\bar X^{RE}_n(t_j))\mathcal{L}p_W(s,W(s))\rd s\Bigl\|.
\end{eqnarray}
Hence, for $i=0,1,\ldots,n$
\begin{eqnarray}
     &&\|\bar e_i\|\leq\sum\limits_{j=0}^{i-1}\|A_j\|+\Bigl\|\sum\limits_{j=0}^{i-1}B_j\Bigl\|+\sum\limits_{j=0}^{n-1}\|C_{1,j}\|\notag\\
     &&\quad\quad\quad+\max\limits_{1\leq i \leq n}\Bigl\|\sum\limits_{j=0}^{i-1}C_{2,j}\Bigl\|+\delta_3\cdot\max\limits_{1\leq i \leq n} C^1_{3,i}+\delta_3 \cdot C^2_{3,n},
\end{eqnarray}
and for all $k=0,1,\ldots,n$
\begin{eqnarray}
\label{error_est_0}
     &&\mathbb{E}\Bigl(\max\limits_{0\leq i\leq k}\|\bar e_i\|^r\Bigr)\leq c_r\mathbb{E}\Bigl(\sum\limits_{j=0}^{k-1}\|A_j\|\Bigr)^r+c_r\mathbb{E}\Bigl(\max\limits_{1\leq i\leq k}\Bigl\|\sum\limits_{j=0}^{i-1}B_j\Bigl\|^r\Bigr)+c_r\mathbb{E}\Bigl(\sum\limits_{j=0}^{n-1}\|C_{1,j}\|\Bigr)^r\notag\\
     &&\quad\quad\quad+c_r\mathbb{E}\Bigl(\max\limits_{1\leq i \leq n}\Bigl\|\sum\limits_{j=0}^{i-1}C_{2,j}\Bigl\|^r\Bigr)+c_r\delta_3^r\cdot\mathbb{E}\Bigl(\max\limits_{1\leq i \leq n} (C^1_{3,i})^r\Bigr)+c_r\delta_3^r \cdot \mathbb{E}(C^2_{3,n})^r.
\end{eqnarray}
By the Jensen inequality we have for $k=0,1,\ldots,n$
\begin{equation}
    \Bigl(\sum\limits_{j=0}^{k-1}\|\bar e_j\|\Bigr)^r\leq k^{r-1}\sum\limits_{j=0}^{k-1}\|e_j\|^r\leq n^{r-1}\sum\limits_{j=0}^{k-1}\|e_j\|^r,
\end{equation}
which implies that
\begin{equation}
    \Bigl(\frac{1}{n}\sum\limits_{j=0}^{k-1}\|\bar e_j\|\Bigr)^r\leq\frac{1}{n}\sum\limits_{j=0}^{k-1}\|e_j\|^r.
\end{equation}
Moreover
\begin{equation}
    \|A_j\|\leq \frac{KT}{n}\|\bar e_j\|,
\end{equation}
and hence 
\begin{equation}
\label{est_A_1}
    \Bigl(\sum\limits_{j=0}^{k-1}\|A_j\|\Bigr)^r\leq K^rT^r\Bigl(\frac{1}{n}\sum\limits_{j=0}^{k-1}\|\bar e_j\|\Bigr)^r\leq\frac{K^rT^r}{n}\sum\limits_{j=0}^{k-1}\|\bar e_j\|^r.
\end{equation}
It holds  that $\displaystyle{\Bigl(\sum\limits_{j=0}^k B_j,\tilde{\Sigma}^n_{t_{k+1}}\Bigr)_{k=0,1,\ldots,n-1}}$ is a discrete-time martingale. To see that let us denote by $\displaystyle{M_k \coloneqq \sum\limits_{j=0}^k B_j}$. By the basic properties of the It\^o integral we get for $k=0,1,\ldots,n-1$ that $\sigma ( M_k ) \subset \tilde{\Sigma}^n_{t_{k+1}}$,
\begin{eqnarray*}
\mathbb{E}(M_{k+1} - M_k | \tilde{\Sigma}^n_{t_{k+1}}) = \mathbb{E}(B_{k+1} | \tilde{\Sigma}^n_{t_{k+1}}) = 0, \quad k=0,1,\ldots,n-2,
\end{eqnarray*}
and for $j=0,1,\ldots,n-1$
\begin{equation}
    \mathbb{E}\|B_j\|^r\leq C(T/n)^{r/2}\mathbb{E}\|\bar e_j\|^r<+\infty.
\end{equation}
Hence, by the Burkholder and Jensen inequalities we get for $k=0,1,\ldots,n$
\begin{eqnarray}
\label{est_B_1}
     &&\mathbb{E}\Bigl(\max\limits_{1\leq i\leq k}\Bigl\|\sum\limits_{j=0}^{i-1}B_j\Bigl\|^r\Bigr)=\mathbb{E}\Bigl(\max\limits_{0\leq i\leq k-1}\Bigl\|\sum\limits_{j=0}^{i}B_j\Bigl\|^r\Bigr)\leq C_r^r\mathbb{E}\Bigl(\sum\limits_{j=0}^{k-1}\|B_j\|^2\Bigr)^{r/2}\notag\\
     &&\quad\quad\quad\quad\quad\quad\quad \quad\quad\quad \leq C_r^r k^{\frac{r}{2}-1}\sum\limits_{j=0}^{k-1}\mathbb{E}\|B_j\|^r\leq \frac{C}{n}\sum\limits_{j=0}^{k-1}\mathbb{E}\|\bar e_j\|^r.
\end{eqnarray}
From \eqref{error_est_0}, \eqref{est_A_1}, \eqref{est_B_1}, and the fact that $\bar e_0=0$ we get for $k=0,1,\ldots,n$ that
\begin{eqnarray}
\label{main_err_est_1}
    &&\mathbb{E}\Bigl(\max\limits_{0\leq i\leq k}\|\bar e_i\|^r\Bigr)\leq \frac{C}{n}\sum\limits_{j=0}^{k-1}\mathbb{E}\|\bar e_j\|^r+c_r R_n\leq\frac{C}{n}\sum\limits_{j=0}^{k-1}\mathbb{E}\Bigl(\max\limits_{0\leq i\leq j}\|\bar e_i\|^r\Bigr)+c_r R_n\notag\\
    &&\quad\quad\quad\quad\quad\quad\quad=\frac{C}{n}\sum\limits_{j=1}^{k-1}\mathbb{E}\Bigl(\max\limits_{0\leq i\leq j}\|\bar e_i\|^r\Bigr)+c_r R_n,
\end{eqnarray}
where
\begin{eqnarray}
   && R_n=\mathbb{E}\Bigl(\sum\limits_{j=0}^{n-1}\|C_{1,j}\|\Bigr)^r+\mathbb{E}\Bigl(\max\limits_{1\leq i \leq n}\Bigl\|\sum\limits_{j=0}^{i-1}C_{2,j}\Bigl\|^r\Bigr)\notag\\
   &&\quad\quad + \delta_3^r\cdot\mathbb{E}\Bigl(\max\limits_{1\leq i \leq n} (C^1_{3,i})^r\Bigr)+\delta_3^r \cdot \mathbb{E}(C^2_{3,n})^r.
\end{eqnarray}
By the discrete version of the Gronwall's lemma (see, for example, Lemma 2.1 in \cite{KRWU_0}) we get 
\begin{equation}
\label{est_bier_1}
    \mathbb{E}\Bigl(\max\limits_{0\leq i\leq n}\|\bar e_i\|^r\Bigr)\leq KR_n.
\end{equation}
By the Jensen inequality and Lemma \ref{est_cont_Euler} we get
\begin{eqnarray}
\label{est_bier_2}
     &&\mathbb{E}\Bigl(\sum\limits_{j=0}^{n-1}\|C_{1,j}\|\Bigr)^r\leq
     \delta_1^r T^r n^{-1} \sum\limits_{j=0}^{n-1}\mathbb{E}\|p_a(\xi_j,\bar X^{RE}_n(t_j))\|^r\notag\\
     &&\leq C\delta_1^r (1+\max\limits_{0\leq j\leq n}\mathbb{E}\|\bar X^{RE}_n(t_j)\|^r)\leq K_1\delta_1^r.
\end{eqnarray}
The process $\displaystyle{\Bigl(\sum\limits_{j=0}^k C_{2,j},\tilde{\Sigma}^n_{t_{k+1}}\Bigr)_{k=0,1,\ldots,n-1}}$ is a discrete-time martingale - this can be justified in analogous way as for $\displaystyle{\Bigl(\sum\limits_{j=0}^k B_j,\tilde{\Sigma}^n_{t_{k+1}}\Bigr)_{k=0,1,\ldots,n-1}}$. Hence, again by the Burkholder and Jensen inequalities, we obtain
\begin{eqnarray}
\label{est_bier_3}
     &&\mathbb{E}\Bigl(\max\limits_{1\leq i \leq n}\Bigl\|\sum\limits_{j=0}^{i-1}C_{2,j}\Bigl\|^r\Bigr)\leq C_r^r n^{\frac{r}{2}-1}\sum\limits_{j=0}^{n-1}\mathbb{E}\|C_{2,j}\|^r\notag\\
     &&\quad\quad\leq K_1 (1+\max\limits_{0\leq j\leq n}\mathbb{E}\|\bar X^{RE}_n(t_j)\|^r)\delta_2^r
     \leq K_2\delta_2^r.
\end{eqnarray}
Let us denote by 
\begin{equation}
    D_j=\int\limits_{t_j}^{t_{j+1}}\tilde b
     (t_j,\bar X^{RE}_n(t_j))\frac{\partial p_W}{\partial y}(s,W(s))\rd W(s),
\end{equation}
then $\displaystyle{\Bigl(\sum\limits_{j=0}^k D_{j},\tilde{\Sigma}^n_{t_{k+1}}\Bigr)_{k=0,1,\ldots,n-1}}$ is also a discrete-time martingale. Therefore, by the Burkholder and Jensen inequalities, we obtain
\begin{eqnarray}
     &&\mathbb{E}\Bigl(\max\limits_{1\leq i \leq n} (C^1_{3,i})^r\Bigr)=\mathbb{E}\Bigl(\max\limits_{1\leq i\leq n}\Bigl\|\sum\limits_{j=0}^{i-1}D_j\Bigl\|^r\Bigr)\leq C_r^r n^{\frac{r}{2}-1}\sum\limits_{j=0}^{n-1}\mathbb{E}\|D_j\|^r,
\end{eqnarray}
where, by \eqref{est_dpdy} and submultiplicativity of the Frobenius norm, we get
\begin{eqnarray}
     &&\mathbb{E}\|D_j\|^r=\mathbb{E}\Bigl\|\int\limits_{t_j}^{t_{j+1}}\tilde b
     (t_j,\bar X^{RE}_n(t_j))\frac{\partial p_W}{\partial y}(s,W(s))\rd W(s)\Bigl\|^r\notag\\
     &&\leq C(T/n)^{\frac{r}{2}-1}\mathbb{E}\int\limits_{t_j}^{t_{j+1}}\Bigl\|\tilde b
     (t_j,\bar X^{RE}_n(t_j))\frac{\partial p_W}{\partial y}(s,W(s))\Bigl\|^r \rd s\notag\\
     &&\leq C(T/n)^{\frac{r}{2}-1}\mathbb{E}\int\limits_{t_j}^{t_{j+1}}\Bigl\|\tilde b
     (t_j,\bar X^{RE}_n(t_j))\Bigl\|^r\cdot\Bigl\|\frac{\partial p_W}{\partial y}(s,W(s))\Bigl\|^r \rd s\leq K_3 n^{-r/2},
\end{eqnarray}
for $j=0,1,\ldots,n-1$. This implies that
\begin{equation}
\label{est_bier_4}
    \mathbb{E}\Bigl(\max\limits_{1\leq i \leq n} (C^1_{3,i})^r\Bigr)\leq K_4.
\end{equation}
Finally, from \eqref{est_lpw} and Lemma \ref{est_cont_Euler}
\begin{eqnarray}
\label{est_bier_5}
     &&\mathbb{E}(C^2_{3,n})^r=\mathbb{E}\Bigl(\sum\limits_{j=0}^{n-1}\Bigl\|\int\limits_{t_j}^{t_{j+1}}\tilde b
     (t_j,\bar X^{RE}_n(t_j))\mathcal{L}p_W(s,W(s))\rd s\Bigl\|\Bigr)^r\notag\\
     &&\leq n^{r-1}\sum\limits_{j=0}^{n-1}\mathbb{E}\Bigl\|\int\limits_{t_j}^{t_{j+1}}\tilde b
     (t_j,\bar X^{RE}_n(t_j))\mathcal{L}p_W(s,W(s))\rd s\Bigl\|^r\notag\\
     &&\leq n^{r-1}\sum\limits_{j=0}^{n-1}\mathbb{E}\Bigl(\int\limits_{t_j}^{t_{j+1}}\|\tilde b
     (t_j,\bar X^{RE}_n(t_j))\|\cdot\|\mathcal{L}p_W(s,W(s))\|\rd s\Bigr)^r\notag\\
     &&\leq \frac{C}{n}\sum\limits_{j=0}^{n-1}\mathbb{E}\|\tilde b
     (t_j,\bar X^{RE}_n(t_j))\|^r\leq K_5.
\end{eqnarray}
Combining  \eqref{est_bier_1}, \eqref{est_bier_2}, \eqref{est_bier_3}, \eqref{est_bier_4}, \eqref{est_bier_5} we obtain 
\begin{equation}
     \mathbb{E}\Bigl(\max\limits_{0\leq i\leq n}\|\bar e_i\|^r\Bigr)\leq K_6(\delta_1^r+\delta_2^r+\delta_3^r),
\end{equation}
which proves (i). 

We now show (ii). Let $\tilde W\in\mathcal{W}_{\alpha,\beta}(\delta_3)$. Note that in this case $Z$ might not be semimartingale nor even a process of bounded variation. Hence, we have that
\begin{equation}
     \bar e_i =  X^{RE}_n(t_i)-\bar X^{RE}_n(t_i)=\sum\limits_{j=0}^{i-1}\Bigl(A_{j}+B_{j}+C_{1,j}+C_{2,j}+C_{3,j}\Bigr),
\end{equation}
where
\begin{eqnarray}
     &&A_{j}=\int\limits_{t_j}^{t_{j+1}}\Bigl(a(\xi_j,X^{RE}_n(t_j))-a(\xi_j,\bar X^{RE}_n(t_j))\Bigr)\rd s\notag\\
     &&B_{j}=\int\limits_{t_j}^{t_{j+1}}\Bigl(b(t_j,X^{RE}_n(t_j))-b(t_j,\bar X^{RE}_n(t_j))\Bigr)\rd W(s)\notag\\
     &&C_{1,j}=(-\delta_1)\int\limits_{t_j}^{t_{j+1}}p_a(\xi_j,\bar X^{RE}_n(t_j))\rd s\notag\\
     &&C_{2,j}=(-\delta_2)\int\limits_{t_j}^{t_{j+1}}p_b(t_j,\bar X^{RE}_n(t_j))\rd W(s),
\end{eqnarray}
and
\begin{equation}
     C_{3,j}=(-\delta_3)\cdot\tilde b(t_j,\bar X^{RE}_n(t_j))\cdot\Delta Z_j.
\end{equation}
Using similar arguments as for the proof of \eqref{main_err_est_1} we have
\begin{equation}
    \mathbb{E}\Bigl(\max\limits_{0\leq i\leq k}\|\bar e_i\|^r\Bigr)\leq \frac{C}{n}\sum\limits_{j=1}^{k-1}\mathbb{E}\Bigl(\max\limits_{0\leq i\leq j}\|\bar e_i\|^r\Bigr)+c_r \bar R_n,
\end{equation}
where this time we get  from \eqref{est_bier_2}, \eqref{est_bier_3}, and Lemma \ref{est_cont_Euler2} that
\begin{eqnarray}
\label{bar_rn_1}
   && \bar R_n=\mathbb{E}\Bigl(\sum\limits_{j=0}^{n-1}\|C_{1,j}\|\Bigr)^r+\mathbb{E}\Bigl(\max\limits_{1\leq i \leq n}\Bigl\|\sum\limits_{j=0}^{i-1}C_{2,j}\Bigl\|^r\Bigr) + \mathbb{E}\Bigl(\max\limits_{1\leq i \leq n}\Bigl\|\sum\limits_{j=0}^{i-1}C_{3,j}\Bigl\|^r\Bigr)\notag\\
   &&\leq C(\delta_1^r+\delta_2^r)(1+\delta_3^r n^{r(1-\gamma)})\cdot e^{C(1+\delta_3^r n^{r(1-\gamma)})}+\mathbb{E}\Bigl(\max\limits_{1\leq i \leq n}\Bigl\|\sum\limits_{j=0}^{i-1}C_{3,j}\Bigl\|^r\Bigr).
\end{eqnarray}
Agian by the discrete version of the Gronwall's lemma we get 
\begin{equation}
\label{est_bier_1}
    \mathbb{E}\Bigl(\max\limits_{0\leq i\leq n}\|\bar e_i\|^r\Bigr)\leq K\bar R_n.
\end{equation}
Moreover,
\begin{equation}
    \Biggl\|\max\limits_{1\leq i \leq n}\Bigl\|\sum\limits_{j=0}^{i-1}C_{3,j}\Bigl\|\Biggl\|_r\leq \Bigl\|\sum\limits_{j=0}^{n-1}\|C_{3,j}\|\Bigl\|_r\leq\sum\limits_{j=0}^{n-1}\Bigl\| \|C_{3,j}\|\Bigl\|_r,
\end{equation}
and, since $\bar X^{RE}_n(t_j)$ and $\Delta W_j$ are independent, we have by Lemma \ref{est_cont_Euler2}
\begin{eqnarray}
    &&\Bigl\| \|C_{3,j}\|\Bigl\|_r\leq C\delta_3\Bigl\|1+\|\bar X^{RE}_n(t_j)\|\Bigl\|_r\cdot\Bigl\|(T/n)^{\alpha}+\|\Delta W_j\|^{\beta}\Bigl\|_r\notag\\
    &&\leq C\delta_3\cdot n^{-\gamma}\cdot\Bigl(1+\max\limits_{0\leq i \leq n}\|\bar X^{RE}_n(t_i)\|_r\Bigr)\notag\\
    &&\leq C \delta_3 n^{-\gamma}\cdot (1+\delta_3 n^{1-\gamma})\cdot e^{C(1+(\delta_3 n^{1-\gamma})^r)},
\end{eqnarray}
and
\begin{equation}
\label{est_c3j}
    \Biggl\|\max\limits_{1\leq i \leq n}\Bigl\|\sum\limits_{j=0}^{i-1}C_{3,j}\Bigl\|\Biggl\|_r\leq C \delta_3 n^{1-\gamma}\cdot (1+\delta_3 n^{1-\gamma})\cdot e^{C(1+(\delta_3 n^{1-\gamma})^r)}.
\end{equation}
From \eqref{bar_rn_1}, \eqref{est_bier_1}, and \eqref{est_c3j} we get the thesis of (i).
\end{proof}
\begin{theorem}
\label{theorem_error_bounds}
Let $r\in [2,+\infty)$, $\varrho\in (0,1]$. 
\begin{itemize}
    \item [(i)] There exists  $C\in (0,+\infty)$, depending only on the parameters of the class $\mathcal{F}(\varrho,K)$ and $r$, such that for all $n\in\mathbb{N}$, $\delta_1,\delta_2,\delta_3\in [0,1]$, $(a,b,\eta)\in\mathcal{F}(\varrho,K)$, $(\tilde a,\tilde b, \tilde W)\in V_a(\delta_1)\times V_b(\delta_2)\times\mathcal{W}_0(\delta_3)$ it holds
\begin{equation}
   \|X(a,b,W,\eta)(T)-\mathcal{\bar A}^{RE}_n(\tilde{a}, \tilde{b},\tilde{W},\eta)\|_r\leq C(n^{-\min\{\varrho,1/2\}}+\delta_1+\delta_2+\delta_3).
\end{equation}
 \item [(ii)] Let $\alpha,\beta\in (0,1]$. There exists  $C\in (0,+\infty)$, depending only on the parameters of the class $\mathcal{F}(\varrho,K)$ and $r$, such that for all $n\in\mathbb{N}$, $\delta_1,\delta_2,\delta_3\in [0,1]$, $(a,b,\eta)\in\mathcal{F}(\varrho,K)$, $(\tilde a,\tilde b, \tilde W)\in V_a(\delta_1)\times V_b(\delta_2)\times\mathcal{W}_{\alpha,\beta}(\delta_3)$ it holds
\begin{eqnarray}
   &&\|X(a,b,W,\eta)(T)-\mathcal{\bar A}^{RE}_n(\tilde{a}, \tilde{b},\tilde{W},\eta)\|_r\leq Cn^{-\min\{\varrho,1/2\}}\notag\\
   &&\quad\quad+C(\delta_1+\delta_2+\delta_3\cdot n^{1-\gamma})\cdot (1+\delta_3 n^{1-\gamma})\cdot e^{C(1+(\delta_3 n^{1-\gamma})^r)},
\end{eqnarray}
 where $\gamma=\min\{\alpha,\beta/2\}$.
\end{itemize}
\end{theorem} 
\begin{proof}
By Proposition 1 in \cite{PMPP17} (for the case $\delta_1=\delta_2=\delta_3=0$) and from Proposition \ref{prop1} (i) we get
\begin{eqnarray}
&&\|X(a,b,W,\eta)(T)-\mathcal{\bar A}^{RE}_n(\tilde{a}, \tilde{b},\tilde{W},\eta)\|_r\leq \|X(a,b,W,\eta)(T)-X_n^{RE}(a,b,W,\eta)(T)\|_r\notag\\
&&+\|X_n^{RE}(a,b,W,\eta)(T)-\bar X_n^{RE}(\tilde a,\tilde b,\tilde W,\eta)(T)\|_r\leq C_1n^{-\min\{\varrho,1/2\}}+C_2(\delta_1+\delta_2+\delta_3),
\end{eqnarray}
which implies (i). Similarly, by using the above error decomposition together with Proposition \ref{prop1} (ii) we obtain the thesis of (ii).
\end{proof}
\section{Lower bounds and optimality of the randomized Euler algorithm}
In this section, we investigate lower error bound for an arbitrary method \eqref{alg_def} from the class $\Phi_n$. We focus only on the class $\mathcal{W}_0$ of disturbed Wiener processes $\tilde W$. Essentially, sharp lower bounds in the class $\mathcal{W}_{\alpha,\beta}$ are left as an open problem. For some special cases we also show optimality of the randomized Euler algorithm $\mathcal{\bar A}^{RE}_n$.

The following result follows from Lemma 3 in \cite{PMPP17} and Theorem \ref{theorem_error_bounds}.
\begin{theorem}
\label{thm_lower_bounds}
Let $r\in [2,+\infty)$, $K\in (0,+\infty)$, $\varrho\in (0,1]$. Then there exist   $C_1,C_2\in (0,+\infty)$, depending only on the parameter of the class $\mathcal{F}(\varrho,K)$ and $r$, such that for all $n\in\mathbb{N}$, $\delta_1,\delta_2,\delta_3\in [0,1]$ it holds
\begin{displaymath}
	C_1(n^{-\min\{\varrho,1/2\}}+\delta_1+\delta_2)\leq e^{(r)}_n(\mathcal{F}(\varrho,K),\mathcal{W}_0,\delta_1,\delta_2,\delta_3)\leq C_2(n^{-\min\{\varrho,1/2\}}+\delta_1+\delta_2+\delta_3).
\end{displaymath}
In particular, the $n$th minimal error satisfies
\begin{equation}
\label{opt_1}
    e^{(r)}_n(\mathcal{F}(\varrho,K),\mathcal{W}_0,\delta_1,\delta_2,0)=\Theta(n^{-\min\{\varrho,1/2\}}+\delta_1+\delta_2),
\end{equation}
and
\begin{equation}
\label{opt_2}
     e^{(r)}_n(\mathcal{F}(\varrho,K),\mathcal{W}_0,\delta_1,\delta_2,\max\{\delta_1,\delta_2\})=\Theta(n^{-\min\{\varrho,1/2\}}+\delta_1+\delta_2),
\end{equation}
as $n\to+\infty$, $\max\{\delta_1,\delta_2\}\to 0^+$. In both cases \eqref{opt_1}, \eqref{opt_2} an optimal algorithm is the randomized Euler algorithm $\mathcal{\bar A}^{RE}_n$. 
\end{theorem}
In general we have a gap between upper and lower bounds, and sharp bounds appear only in special cases (i.e.: when $\delta_3=0$ or $\delta_3=\max\{\delta_1,\delta_2\}$). However, in the particular case for the randomized Euler algorithm we have the following bounds for its worst-case error (the proof follows from Proposition 1 in \cite{AKPMPP}).
\begin{proposition}
Let $r\in [2,+\infty)$, $K\in (0,+\infty)$, $\varrho\in (0,1]$. Then for the randomized Euler algorithm $\mathcal{\bar A}^{RE}_n$ it holds
\begin{displaymath}
	 e^{(r)}(\mathcal{\bar A}^{RE}_n,\mathcal{F}(\varrho,K),\mathcal{W}_0,\delta_1,\delta_2,\delta_3)=\Theta(n^{-\min\{\varrho,1/2\}}+\delta_1+\delta_2+\delta_3),
\end{displaymath}
as $n\to+\infty$, $\max\{\delta_1,\delta_2,\delta_3\}\to 0+$.
\end{proposition}

\section{Numerical experiments}
\def\mW{m}
\def\d{\mathrm{d}} 

Let us consider the following linear SDE that describes the well-known multidimensional Black-Scholes model
\begin{equation}
	\label{PROBLEM_LIN_SDE_MD}
		\left\{ \begin{array}{ll}
			\d X(t)=\begin{pmatrix}\mu_1 X_1(t)\\
   \mu_2 X_2(t)\\ \vdots \\ \mu_d X_d(t) \end{pmatrix}\d t+\begin{pmatrix}
\sigma^{1,1}X_1(t) &\sigma^{1,2}X_1(t)  & \cdots & \sigma^{1, \mW}X_1(t)\\ 
\sigma^{2,1}X_2(t) &\sigma^{2,2}X_2(t)  & \cdots & \sigma^{2, \mW}X_2(t)\\ 
\vdots  & \vdots  & \ddots  & \vdots \\ 
\sigma^{d,1}X_d(t) &\sigma^{d,2}X_d(t)  & \cdots & \sigma^{d, \mW}X_d(t) 
\end{pmatrix}\d W(t), & \\
			X(0)=x_0,  \quad t\in [0,T],
		\end{array}\right.
\end{equation}
where $\sigma^{i,j}>0$ for $i\in\{1,\ldots, d\}$, $j\in\{1,\ldots, \mW\}$, $\mu_i\in\R$ , $x_0\in\R_{+}^d$. 
Functions $a$ and $b$  take the following forms
$a(t,x)= (\mu_1 x_1,\ldots,\mu_d x_d)^{\text{T}},
b(t,x)=\begin{pmatrix}
\sigma^{1,1}x_1 &\sigma^{1,2}x_1  & \cdots & \sigma^{1, \mW}x_1\\ 
\sigma^{2,1}x_2 &\sigma^{2,2}x_2  & \cdots & \sigma^{2, \mW}x_2\\ 
\vdots  & \vdots  & \ddots  & \vdots \\ 
\sigma^{d,1}x_d &\sigma^{d,2}x_d  & \cdots & \sigma^{d, \mW}x_d 
\end{pmatrix}.$ 
The exact solution of problem \eqref{PROBLEM_LIN_SDE_MD} has the following form

\begin{equation*}
	\label{exact_solM_md}
X_i (t)=X_i(0) \cdot\exp\Biggl(\Biggl(\mu-\frac{1}{2}\sum_{j=1}^{\mW } \Big(\sigma^{i,j}\Big)^2\Biggr)t+\sum_{j=1}^{\mW} \sigma^{i,j} W_j(t)\Biggr)
\end{equation*}
for $i=1,\ldots,d.$

To perform numerical experiments, we choose two examples.

\textbf{Example 1}

	$a(t,x)= \begin{pmatrix}
  0.5 x_1 ,0.7 x_2\\
  \end{pmatrix} ^ T,$ 
  $\ b(t,x)=\begin{pmatrix}
  0.5 x_1 , 0.7 x_1 , 0.2 x_1\\
       -0.5 x_2 , -0.7 x_2   , -0.2 x_2\\
  \end{pmatrix}$,
  $x_0 = (1,2)^T$, $T=1$.

\textbf{Example 2}

	$a(t,x)=\begin{pmatrix}
  0.5 x_1 ,0.7 x_2, 0.4 x_3\\
  \end{pmatrix} ^ T,$ 
  $\ b(t,x)=\begin{pmatrix}
  0.5 x_1 , 0.7 x_1 , 0.2 x_1\\
       0.1 x_2 , 0. x_2   , 0.013x_2\\
       0.  x_3  , 0.75 x_3, 0.013 x_3\\
  \end{pmatrix}$,
  $x_0 = (1,0.1,0.4)^T$, $T=1$.
 We take an estimator of the error of $\|X(T)-\bar X_{n}^{RE}(T)\|_{2}$
\begin{equation}
\label{est_err_1}
	\varepsilon_K(\bar X^{RE}_{n}(T))=\Biggl(\frac{1}{K}\sum_{j=1}^K \|X_{(j)}(T)-\-\bar X^{RE}_{(j),n}(T)\|^2\Biggr)^{1/2}.
\end{equation} 
We also conduct numerical experiments for an equation in which  the exact solution is unknown (Example 3). In this case, to estimate the error $\|X(T)-\bar X_{n}^{RE}(T)\|_{2}$, the exact solution $X(T)$ is approximated by $\bar X^{RE}_{n}$ computed under exact information with for $n=1310720=10\cdot2^{17}$, and then
\begin{equation}
\label{est_err_2}
	\varepsilon_K(\bar X^{RE}_{n}(T))=\Biggl(\frac{1}{K}\sum_{j=1}^K \|\bar X^{RE}_{(j),1310720}(T)-\-\bar X^{RE}_{(j),n}(T)\|^2\Biggr)^{1/2}.
\end{equation} 

\textbf{Example 3}
	$a(t,x)= 0.5\begin{pmatrix}
  t  \sin (10 x_1)\\
  \cos (7x_2)
  \end{pmatrix} ^ T,$ 
  $\ b(t,x)=\begin{pmatrix}
  tx_1 & tx_2 & \sin(x_2) \\
  t \cos (x_1) & x_2 & -x_1
  \end{pmatrix}$,
  $x_0 = (1,2)^T$, $T=1$.

For Examples 1 - 3 we used $K=20000$.

\subsection{Linear disturbing function}

To obtain results in the numerical simulations related to the Theorem \ref{theorem_error_bounds} (ii) we propose the following disruptive function
 $$p_a(t,x)= \textrm{U}_1 \cdot a(t,x)$$ 
 and $$p_b(t,x)= \textrm{U}_2 \cdot b(t,x),$$
 where $\textrm{U}_1,\textrm{U}_2$ has uniform distribution over the interval $[-1,1]$.

As a corrupting function for the Wiener process $W$ in Examples 1 - 3  we take
\begin{equation}
\label{disruptive_wiener_ex1-3}
    p_w(t,x) =\textrm{U}_3\cdot x,
\end{equation}
 where $\textrm{U}_3$ is a random variable with a uniform distribution over the interval $[-1,1]$. We also assume that $\textrm{U}_1$, $\textrm{U}_2$, $\textrm{U}_3$ are independent of $W$ and $\mathcal{G}^n$. We use those uniform distributions to better approximate the worst case scenario setting of error.

\begin{figure*}[h!]
    \centering
    \includegraphics[width=1.0\textwidth]{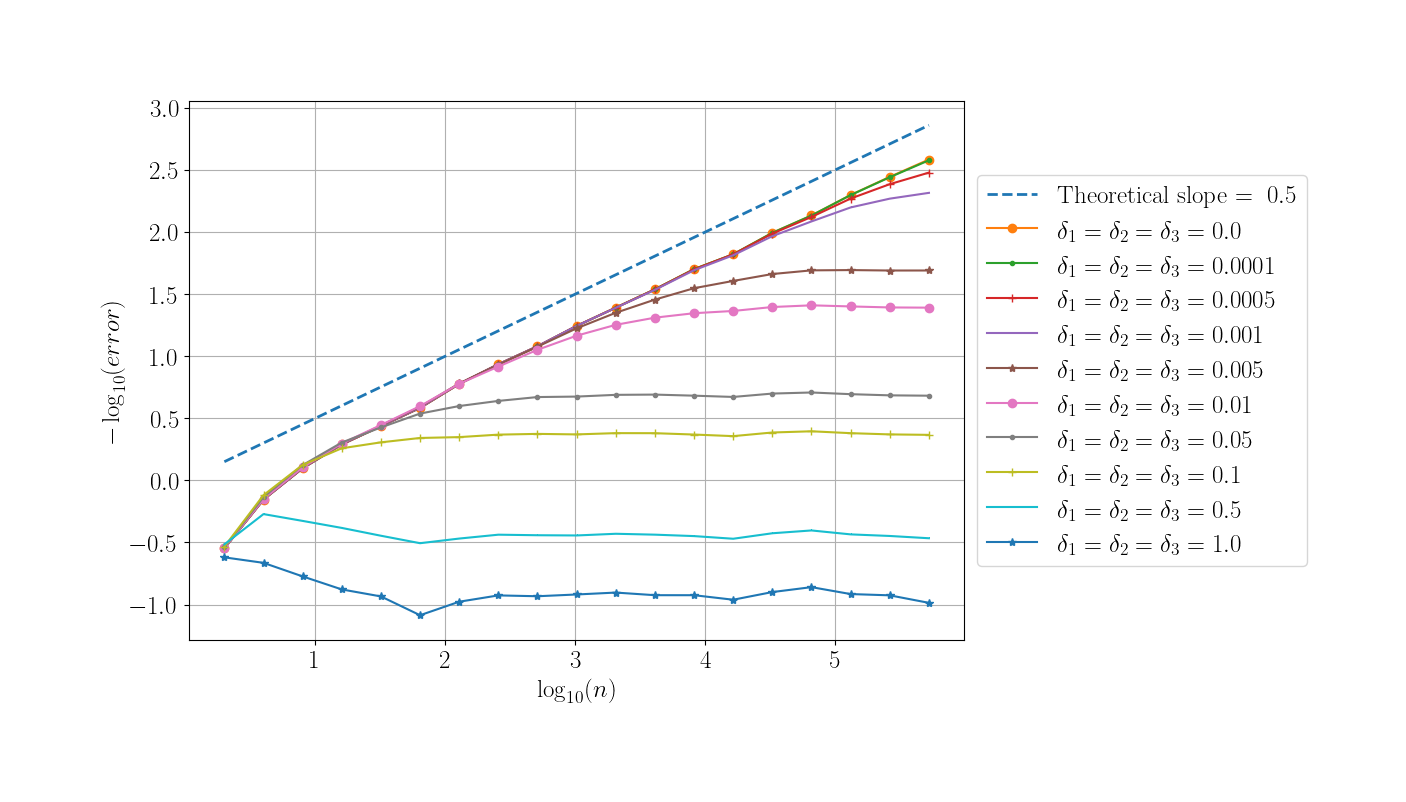}
    \caption{Log-log plot of the randomized error of the Euler algorithm for the Example 1 with $p_w$ as disruption of the Wiener process}
    \label{fig:example1:w0}
\end{figure*}

\begin{figure*}[h!]
    \centering
    \includegraphics[width=1.0\textwidth]{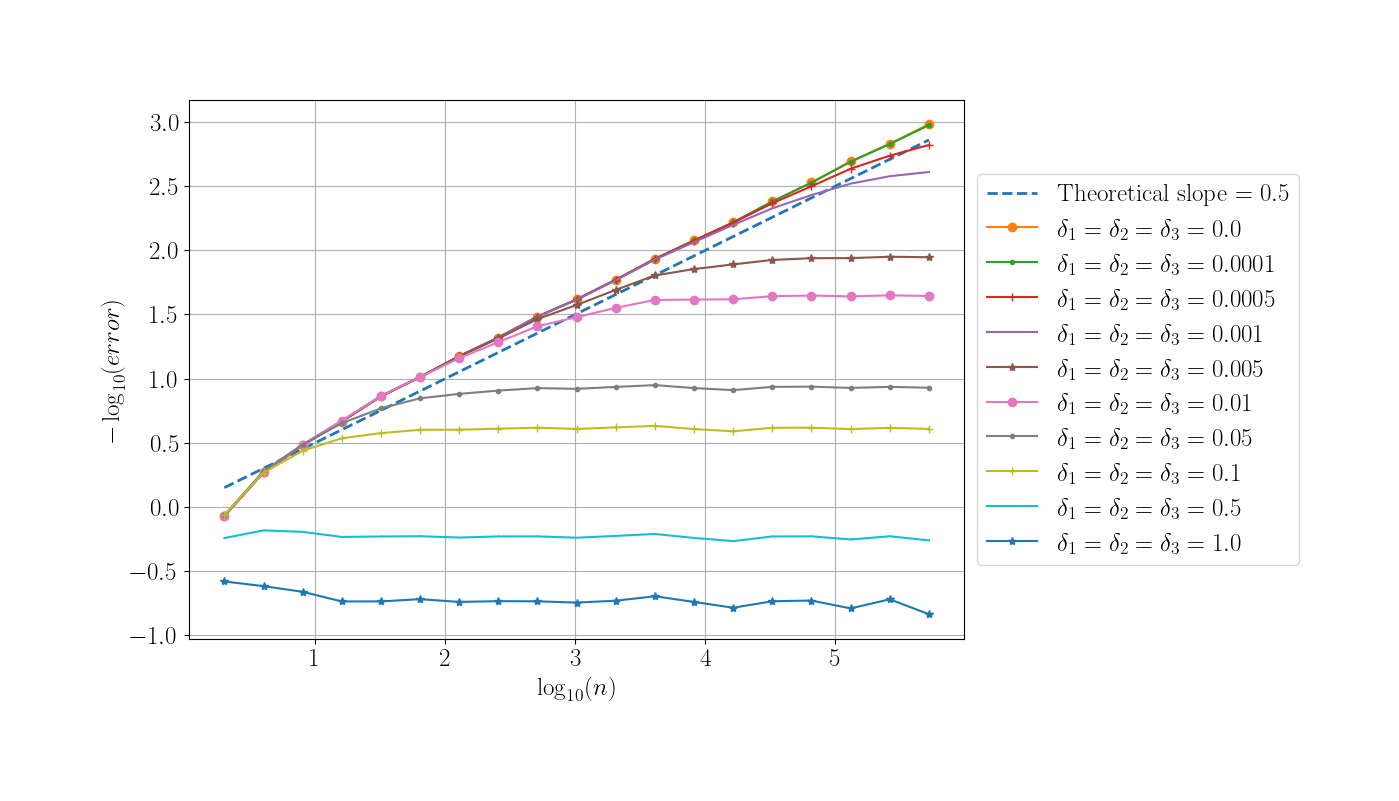}
    \caption{Log-log plot of the randomized error of the Euler algorithm for the Example 2 with $p_w$ as disruption of the Wiener process}
    \label{fig:example2:w0}
\end{figure*}

\begin{figure*}[h!]
    \centering
    \includegraphics[width=1.0\textwidth]{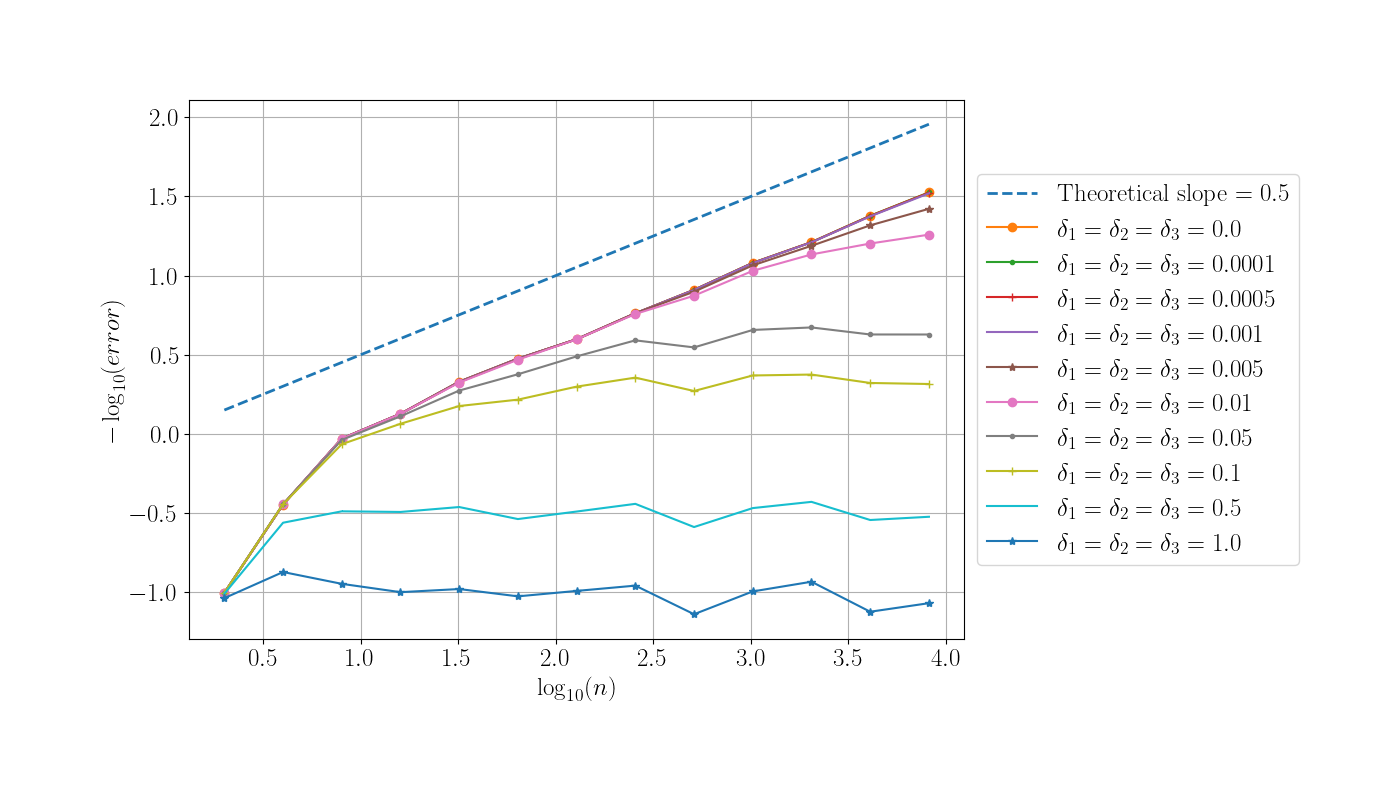}
    \caption{Log-log plot of the randomized error of the Euler algorithm for the Example 3 with $p_w$ as disruption of the Wiener process}
    \label{fig:example3:w0}
\end{figure*}

\subsection{Nonlinear disturbing function}
To conduct illustrative numerical simulations, as per Theorem \ref{theorem_error_bounds} (ii), we propose the following disruptive functions
 $$p_a(t,x)= \textrm{U}_1 \cdot a(t,x)$$ and $$p_b(t,x)= \textrm{U}_2 \cdot b(t,x),$$
 where $\textrm{U}_1,\textrm{U}_2$ is a random variable with a uniform distribution over the interval $[-1,1]$.

As a corrupting function for the Wiener process, we consider
\begin{equation}
\label{disruptive_wiener_ex3}
    p_{w, \beta}(t,x) =\textrm{U}_3 \cdot\text{sgn}(\sin({100\Vert x\Vert}))\cdot \vert\sin({100\Vert x\Vert})\vert^\beta,
\end{equation}
 where $\textrm{U}_3$ is a random variable with a uniform distribution over the interval $[-1,1]$. We also assume that $\textrm{U}_1$, $\textrm{U}_2$, $\textrm{U}_3$ are independent of $W$ and $\mathcal{G}^n$.

\begin{figure*}[h!]
    \centering
    \includegraphics[width=1.0\textwidth]{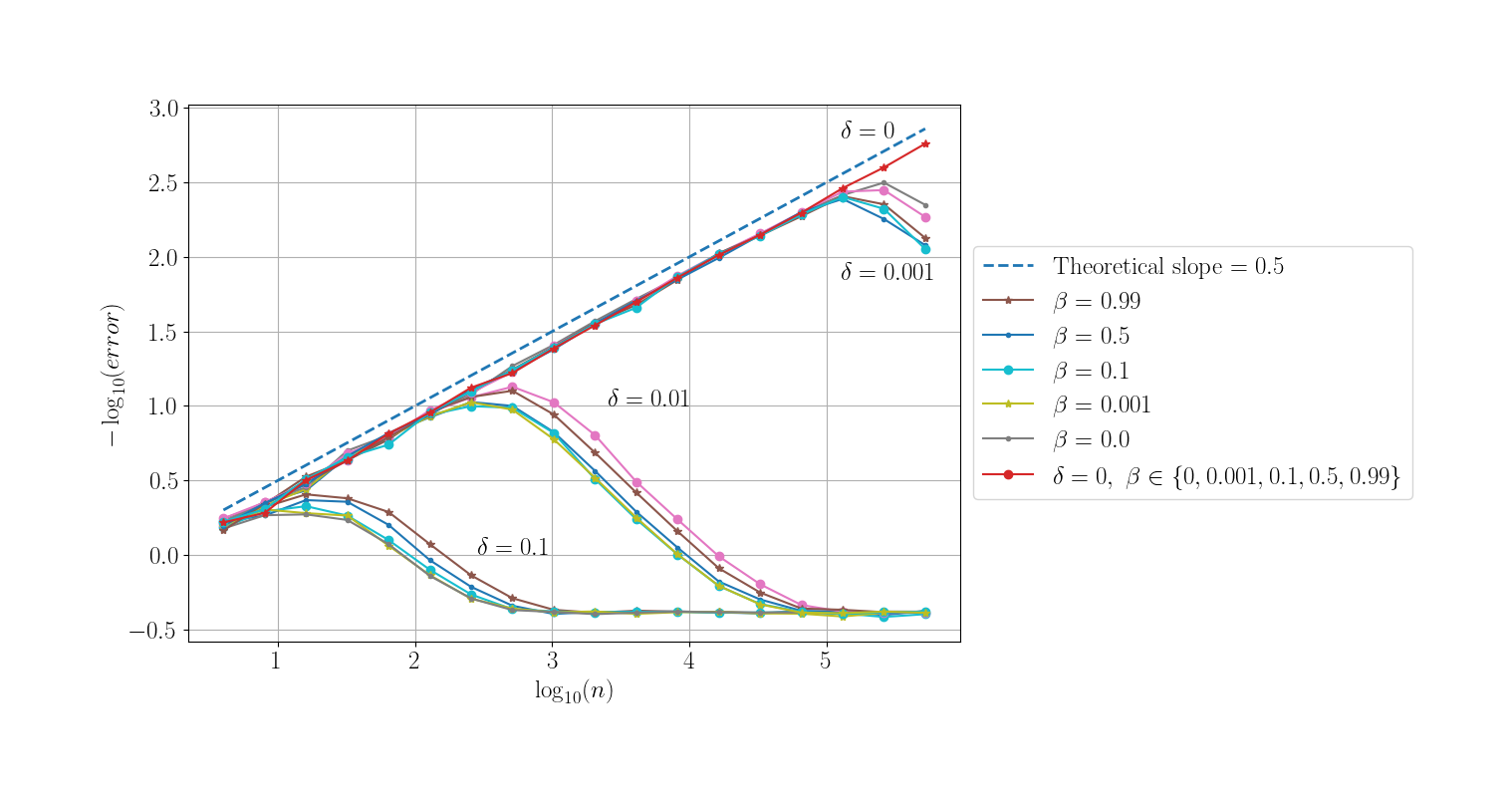}
    \caption{Example 1 with $\delta := \delta_1 = \delta_2 = \delta_3  \in\{0,\ 0.1,\ 0.01,\ 0.0001\}$ and $p_{w,\beta}$ disruption.}
    \label{fig:example1:wb}
\end{figure*}

\begin{figure*}[h!]
    \centering
    \includegraphics[width=1.0\textwidth]{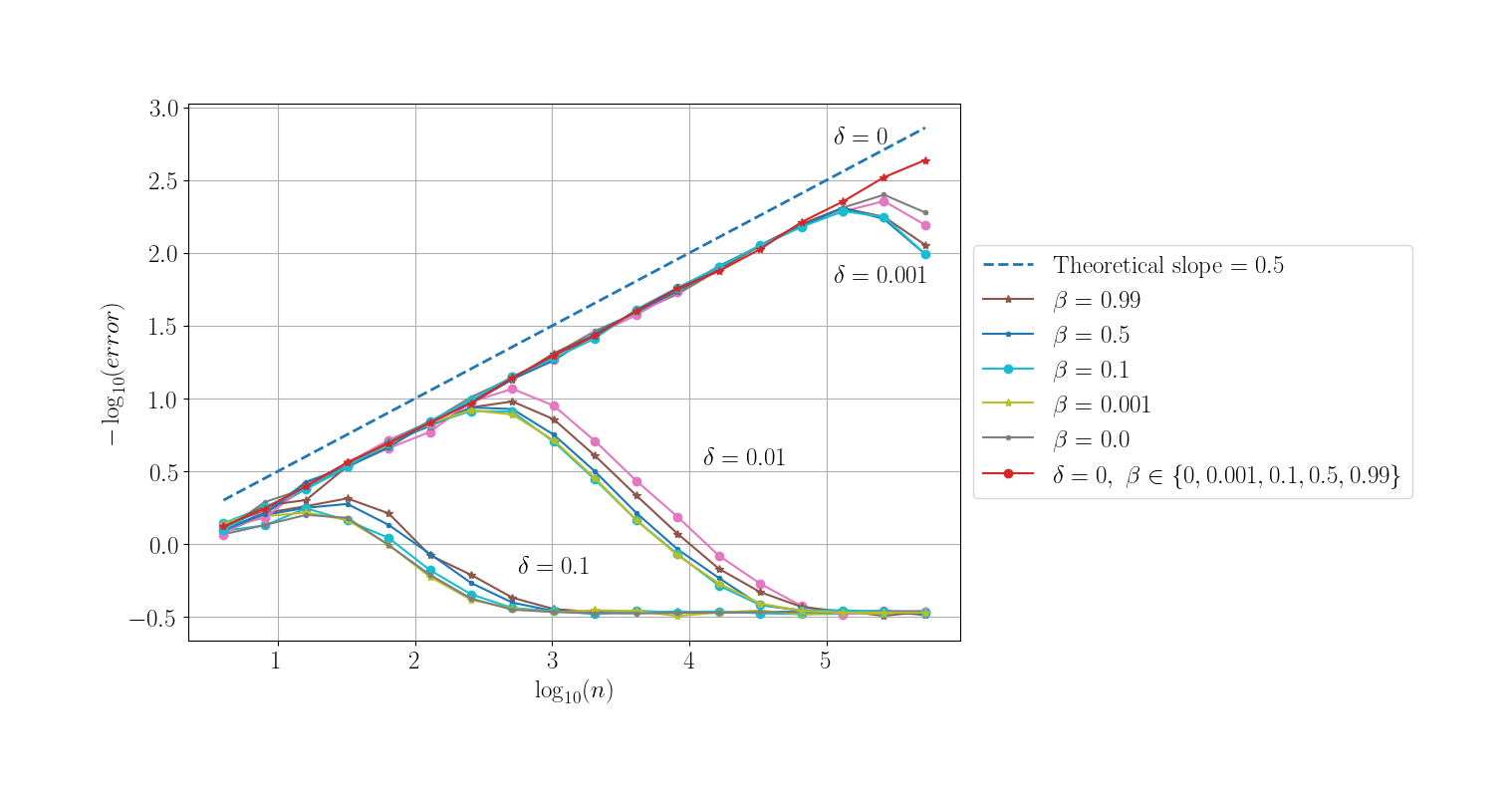}
    \caption{Example 2 with $\delta := \delta_1 = \delta_2 = \delta_3  \in\{0,\ 0.1,\ 0.01,\ 0.0001\}$ and $p_{w,\beta}$ disruption.}
    \label{fig:example2:wb}
\end{figure*}

\begin{figure*}[h!]
    \centering
    \includegraphics[width=1.0\textwidth]{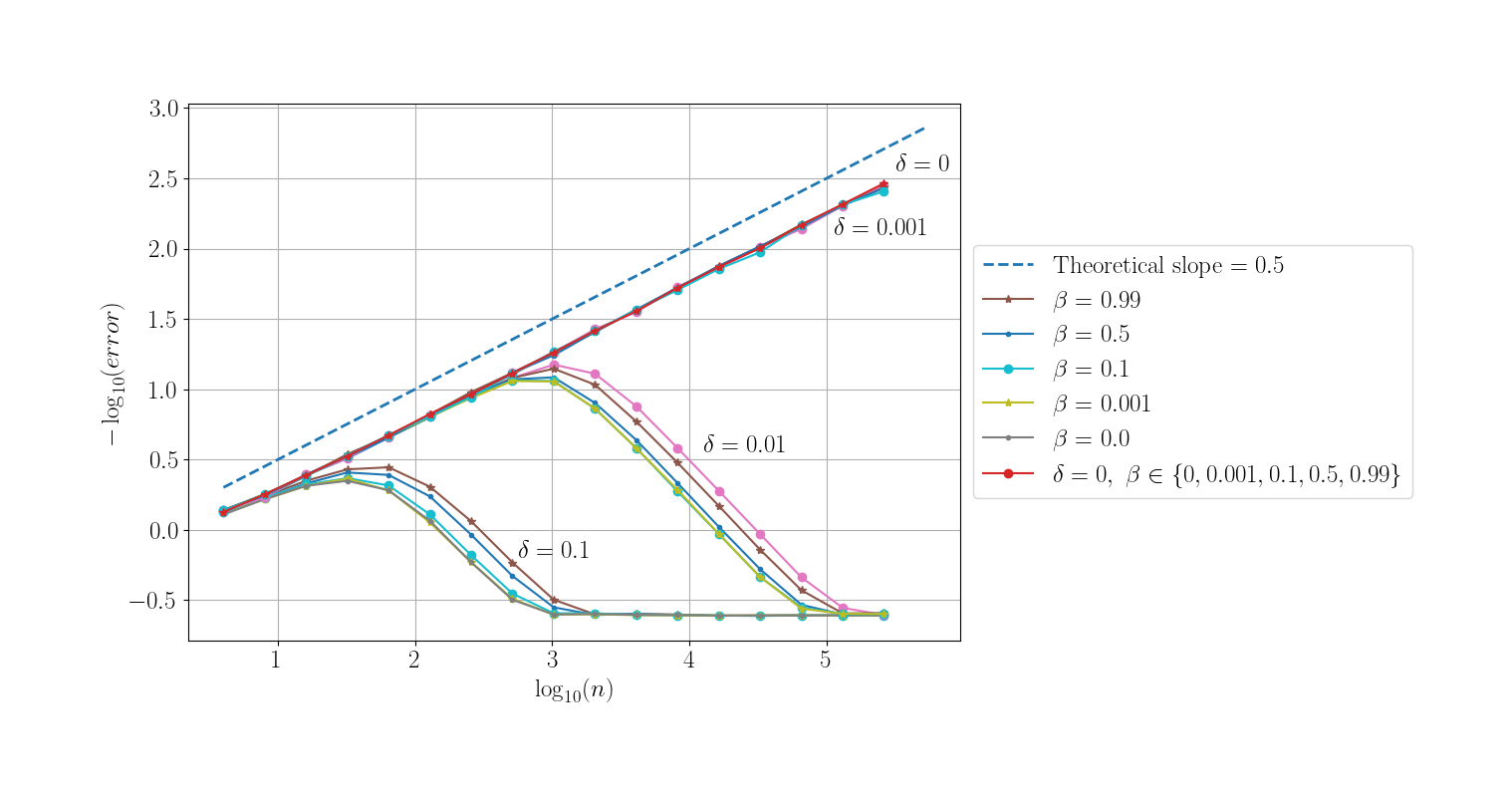}
    \caption{Example 3 with $\delta := \delta_1 = \delta_2 = \delta_3  \in\{0,\ 0.1,\ 0.01,\ 0.0001\}$ and $p_{w,\beta}$ disruption.}
    \label{fig:example3:wb}
\end{figure*}

\begin{figure}[h!]
\includegraphics[width=1.0\textwidth]{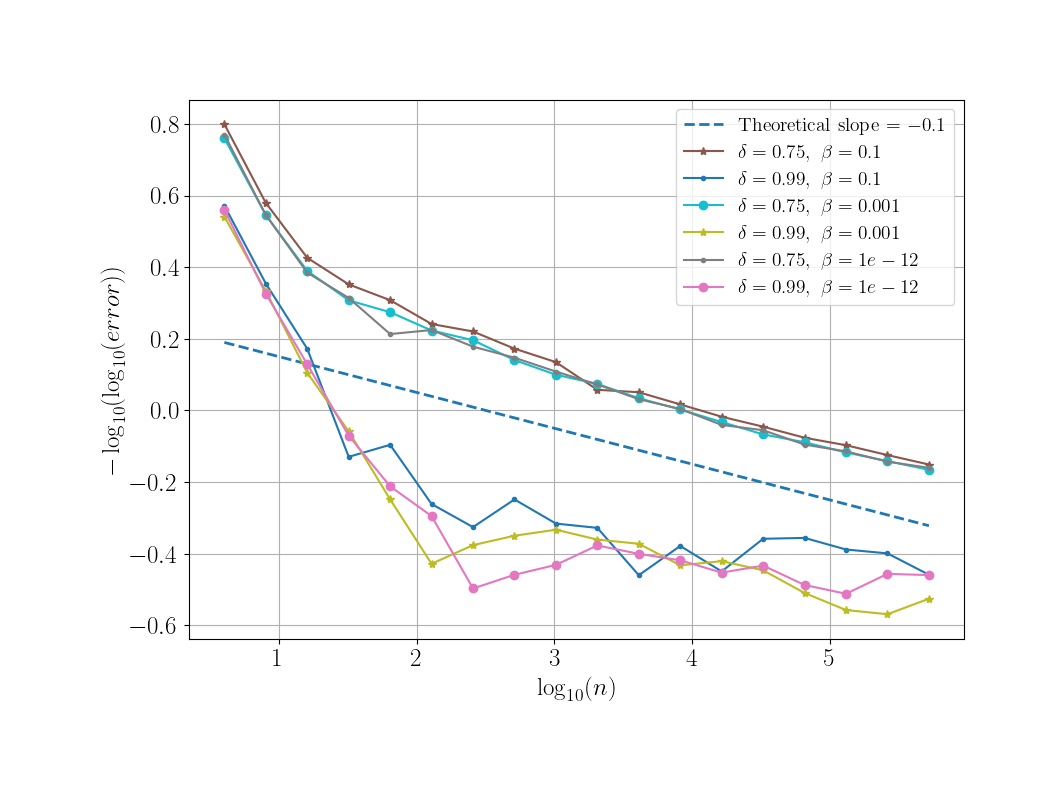}
    \caption{Log-log plot with a doubly logarithmic y-axis for 
Example 3 with \eqref{disruptive_wiener_ex3} as the disruptive function for the Wiener process, with $\delta=0.8,\ 0.9$.}
\label{example3_loglog}
\end{figure}

In Figures \ref{fig:example3:wb} and \ref{example3_loglog}, we present results for \eqref{disruptive_wiener_ex3} as the disruptive function for Wiener process in Example 3.

In this case the, logarithmic error exhibits exponential growth, necessitating the use of doubly logarithmic y-axis. Notably, such error behavior was not observed for disruptive functions from the class $\mathcal{K}_0$ for the Wiener process.




\section{Conclusions}
We have investigated the error and optimality of the randomized Euler scheme in the case when we have access only to noisy standard information about the coefficients $a$, $b$, and driving Wiener process $W$. We considered two classes of disturbed Wiener processes for which we derived upper error bounds for the randomized Euler algorithm. These bounds indicate that the error significantly depends on the regularity of the disturbing functions.
The numerical experiments demonstrate that beyond a certain value of $n$, which depends on the size of the disturbance, the error of the randomized Euler algorithm stabilizes, and increasing number of discretization points $n$ does not lead to reduction in error.

One particularly interesting observation is depicted in Figure \ref{example3_loglog}. When using function \eqref{disruptive_wiener_ex3} as a perturbation for the Wiener process with sufficiently high $\delta$, we observe an exponential increase in error as $n$ increases.

In future research, we plan to investigate the error of (multilevel) Monte Carlo method under inexact information for the weak approximation of solutions of SDEs.
\section{Appendix}
The proof of the following fact is straightforward and, therefore, omitted.
\begin{fact} If $p\in\mathcal{K}_0$ then for all $t\in [0,T],x,y\in\mathbb{R}^m$ it holds
\begin{equation}
    \|p(t,x)\|\leq m^{1/2}(1+\|x\|), \quad
    \|p(t,x)-p(t,y)\|\leq m^{1/2}\|x-y\|,
\end{equation}
\begin{equation}
\label{est_dpdy}
    \Bigl\|\frac{\partial p}{\partial y}(t,y)\Bigl\|\leq m^{1/2},
\end{equation}
\begin{equation}
\label{est_lpw}
    \|\mathcal{L}p(t,y)\| \leq \Bigl( 2m+\frac{m^2}{2}\Bigr)^{1/2}.
\end{equation}
\end{fact}
In order to estimate absolute moments of $\bar X^{RE}_n(t_i)$, in the case when the disturbed Wiener process $\tilde W$ belongs to the class $\mathcal{W}_0(\delta_3)$, we use the following time-continuous randomized Euler process
\begin{equation}
	\label{EU_NO}
		\left\{ \begin{array}{ll}
			\tilde{\bar X}^{RE}_n(0)  & =   \eta, \\
			\tilde{\bar X}^{RE}_n(t)  & = \tilde{\bar X}^{RE}_n(t_i) + \tilde a(\xi_i, \tilde{\bar X}^{RE}_n(t_i)) \cdot (t-t_i) +  \tilde b(t_i, \tilde{\bar X}^{RE}_n(t_i))  \cdot (\tilde W(t)-\tilde W(t_i)) ,
		\end{array}\right.
\end{equation}
for $t\in [t_i,t_{i+1}]$, $i=0,1, \ldots, n-1$, where $\tilde W(t)= W(t)+\delta_3\cdot Z(t)$, $Z(t)=p_W(t, W(t))$, $p_W\in\mathcal{K}_0$. It is easy to see that $\tilde{\bar X}^{RE}_n(t_i)=\bar X^{RE}_n(t_i)$ for $i=0,1,\ldots,n$. Moreover, the process $(\tilde{\bar X}^{RE}_n(t))_{t\in [0,T]}$ is adapted to $(\tilde\Sigma^n_t)_{t\in [0,T]}$, which can be shown by induction.
\begin{lemma}
\label{est_cont_Euler}
Let $r\in [2,+\infty)$. There exists $C\in (0,+\infty)$, depending only on the parameters of the class $\mathcal{F}(\varrho,K)$ and $r$, such that for all $n\in\mathbb{N}$, $\delta_1,\delta_2,\delta_3\in [0,1]$, $(a,b,\eta)\in\mathcal{F}(\varrho,K)$, $(\tilde a,\tilde b, \tilde W)\in V_a(\delta_1)\times V_b(\delta_2)\times\mathcal{W}_0(\delta_3)$ it holds
    \begin{equation}
    \label{est_mom_b_1}
    \sup\limits_{0\leq t \leq T}\mathbb{E}\|\tilde{\bar X}^{RE}_{n}(t)\|^r\leq C(1+\delta_1^r+\delta_2^r+\delta_3^r)e^{CT(1+\delta_1^r+\delta_2^r+\delta_3^r)}.
\end{equation}
\end{lemma}
\begin{proof} We denote by
\begin{equation}
    \bar V_i=(\xi_i, \tilde{\bar X}^{RE}_n(t_i)), \quad \bar U_i=(t_i, \tilde{\bar X}^{RE}_n(t_i)).
\end{equation}
Firstly, we show by induction that
\begin{equation}
\label{est_mom_ran_eu_1}
    \max\limits_{0\leq i\leq n}\mathbb{E}\|\tilde{\bar X}^{RE}_n(t_i)\|^r<+\infty.
\end{equation}
Let us assume that there exists $l\in\{0,1,\ldots,n-1\}$ such that $\max\limits_{0\leq i\leq l}\mathbb{E}\|\tilde{\bar X}^{RE}_n(t_i)\|^r<+\infty$. (This obviously holds for $l=0$.) Due to the fact that $\sigma(\bar U_l)\subset\tilde \Sigma^n_{t_l}$ and, by \eqref{est_lpw},
\begin{equation}
    \|V(t_{l+1})-V(t_{l})\|\leq\int\limits_{t_l}^{t_{l+1}}\|\mathcal{L}p_W(s,W(s))\|\rd s\leq c(m)\cdot (t_{l+1}-t_l),
\end{equation}
we get
\begin{eqnarray}
   && \mathbb{E}\|\tilde{\bar X}^{RE}_n(t_{l+1})\|^r\leq C  \mathbb{E}\|\tilde{\bar X}^{RE}_n(t_{l})\|^r+C(T/n)^r\cdot\mathbb{E}\|\tilde a(\bar U_l)\|^r\notag\\
   &&\quad\quad\quad\quad\quad\quad +\ C\mathbb{E}\|\tilde b(\bar U_l)\|^r\cdot\mathbb{E}\|W(t_{l+1})-W(t_{l})\|^r\\
   &&\quad\quad\quad\quad\quad+\ C \delta_3\mathbb{E}\Bigl\|\int\limits_{t_l}^t\tilde b(\bar U_l)\frac{\partial p_W}{\partial y}(s,W(s))\rd W(s)\Bigl\|^r\notag\\
   &&\quad\quad\quad\quad\quad\ + \delta_3\mathbb{E}\Bigl(\|\tilde b(\bar U_l)\|^r\cdot \|V(t_{l+1})-V(t_l)\|^r\Bigr)\leq K(1+\mathbb{E}\|\tilde{\bar X}^{RE}_n(t_{l})\|^r) <\infty.
\end{eqnarray}
Hence, $\displaystyle{\max\limits_{0\leq i\leq l+1}\mathbb{E}\|\tilde{\bar X}^{RE}_n(t_i)\|^r=\max\{\max\limits_{0\leq i\leq l}\mathbb{E}\|\tilde{\bar X}^{RE}_n(t_i)\|^r,\mathbb{E}\|\tilde{\bar X}^{RE}_n(t_{l+1})\|^r\}<+\infty}$ and the inductive step is completed. Hence, we have shown \eqref{est_mom_ran_eu_1}. Moreover, by \eqref{est_mom_ran_eu_1} we get
\begin{equation}
\label{est_mom_rand_eu_dep_n}
    \sup\limits_{0\leq t\leq T}\mathbb{E}\|\tilde{\bar X}^{RE}_n(t)\|^r\leq C\Bigl(1+\max\limits_{0\leq i\leq n-1}\mathbb{E}\|\tilde{\bar X}^{RE}_n(t_i)\|^r \Bigr)<+\infty
\end{equation}
The constant in \eqref{est_mom_rand_eu_dep_n} depends on $n$. In the second part of the proof we will
show that we can obtain the bound \eqref{est_mom_b_1} with $C$ that does not depend on $n$.

Let for $t\in [0,T]$
\begin{equation}
    \phi_n(t)=\sum\limits_{i=0}^{n-1}\tilde a(\bar V_i)\cdot\mathbf{1}_{(t_i,t_{i+1}]}(t),
\end{equation}
\begin{equation}
    \psi_n(t)=\sum\limits_{i=0}^{n-1}\tilde b(\bar U_i)\cdot\mathbf{1}_{(t_i,t_{i+1}]}(t).
\end{equation}
Note that $\{\psi_n(t)\}_{t\in [0,T]}$ is $\{\tilde\Sigma^n_t\}_{t\geq 0}$-progressively measurable simple process. Hence, we have for all $t\in [0,T]$ that
\begin{equation}
    \label{continous_decomposition}
    \tilde{\bar X}^{RE}_{n}(t)=\eta+\tilde{\bar A}^{RE}_{n}(t)+\tilde{\bar B}^{RE}_{n}(t)+\tilde{\bar C}^{RE}_{n}(t),
\end{equation}
where
\begin{eqnarray}
     && \tilde{\bar A}^{RE}_{n}(t)=\int\limits_0^t\phi_n(s)\rd s,\notag\\
     && \tilde{\bar B}^{RE}_{n}(t)=\int\limits_0^t\psi_n(s)\rd W(s),\notag
\end{eqnarray}
and
\begin{eqnarray}
     && \tilde{\bar C}^{RE}_{n}(t)=\delta_3\cdot\int\limits_0^t\psi_n(s)\rd Z(s)\notag\\
     &&\quad\quad\quad\quad=\delta_3\cdot\int\limits_0^t \psi_n(s)\frac{\partial p_W}{\partial y}(s,W(s))\rd W(s)+\delta_3\cdot\int\limits_0^t\psi_n(s)\mathcal{L}p_W(s,W(s))\rd s.
\end{eqnarray}
and all above stochastic integrals are well-defined. Hence, we have for $t\in [0,T]$
\begin{eqnarray}
     &&\tilde{\bar X}^{RE}_n(t)=\eta+\int\limits_0^t\Bigl[\phi_n(s)+\delta_3\cdot\psi_n(s)\cdot\mathcal{L}p_W(s,W(s))\Bigr]\rd s\notag\\
     &&+\int\limits_0^t\psi_n(s)\cdot\Bigl[I+\delta_3\cdot\frac{\partial p_W}{\partial y}(s,W(s))\Bigr]\rd W(s),
\end{eqnarray}
where $I$ is an identity matrix of size $m\times m$. Hence, by \eqref{est_dpdy}, \eqref{est_lpw}
\begin{eqnarray}
    &&\mathbb{E}\|\tilde{\bar X}^{RE}_n(t)\|^r\leq C_1\|\eta\|^r+C_2\mathbb{E}\int\limits_0^t\|\phi_n(s)\|^r\rd s+C_3(1+\delta_3^r)\mathbb{E}\int\limits_0^t\|\psi_n(s)\|^r\rd s\notag\\
    &&\leq K_1\cdot(1+\|\eta\|^r)\cdot (1+\delta_1^r+\delta_2^r+\delta_3^r)\notag\\
    &&\quad\quad\quad\quad+K_2\cdot(1+\delta_1^r+\delta_2^r+\delta_3^r)\cdot\int\limits_0^t\sum\limits_{i=0}^{n-1}\mathbb{E}\|\tilde{\bar X}^{RE}_n(t_i)\|^r\cdot\mathbf{1}_{(t_i,t_{i+1}]}(s)\rd s
\end{eqnarray}
and therefore for all $t\in [0,T]$
\begin{eqnarray}
    &&\sup\limits_{0\leq u\leq t}\mathbb{E}\|\tilde{\bar X}^{RE}_n(u)\|^r\leq K_1\cdot(1+\|\eta\|^r)\cdot (1+\delta_1^r+\delta_2^r+\delta_3^r)\notag\\
    &&+K_2\cdot (1+\delta_1^r+\delta_2^r+\delta_3^r)\int\limits_0^t\sup\limits_{0\leq u\leq s}\mathbb{E}\|\tilde{\bar X}^{RE}_n(u)\|^r\rd s
\end{eqnarray}
where $K_1$ and $K_2$ depends only on the parameters of the class $\mathcal{F}(\varrho,K)$ and $r$. Since the function $[0,T]\ni t\to \sup\limits_{0\leq u\leq t}\mathbb{E}\|\tilde{\bar X}^{RE}_n(u)\|^r$ is bounded (by \eqref{est_mom_rand_eu_dep_n}) and Borel measurable (as a nondecreasing function), by using the Gronwall's lemma we get \eqref{est_mom_b_1}.
\end{proof}
In the case of the class $\mathcal{W}_{\alpha,\beta}(\delta_3)$ we have the following absolute moments estimate for $\bar X_n^{RE}(t_i)$. The proof technique is different from the one used in the proof of Lemma \ref{est_cont_Euler}, since for $p_W\in\mathcal{K}_{\alpha,\beta}$ the process $Z(t)=p_W(t,W(t))$ might not be a semimartingale nor a process of bounded variation.
\begin{lemma}
    \label{est_cont_Euler2}
Let $r\in [2,+\infty)$. There exists $C\in (0,+\infty)$, depending only on the paramters of the class $\mathcal{F}(\varrho,K)$ and $r$, such that for all $n\in\mathbb{N}$, $\delta_1,\delta_2,\delta_3\in [0,1]$, $(a,b,\eta)\in\mathcal{F}(\varrho,K)$, $(\tilde a,\tilde b, \tilde W)\in V_a(\delta_1)\times V_b(\delta_2)\times\mathcal{W}_{\alpha,\beta}(\delta_3)$ it holds
    \begin{equation}
    \label{est_mom_b_2}
    \mathbb{E}\Bigl[\max\limits_{0\leq i \leq n}\|\bar X_n^{RE}(t_i)\|^r\Bigr]\leq C(1+\delta_3^r n^{r(1-\gamma)})\cdot e^{C(1+\delta_3^r n^{r(1-\gamma)})},
\end{equation}
where $\gamma=\min\{\alpha,\beta/2\}$.
\end{lemma}
\begin{proof}
    We have that  for $i=0,1,\ldots,n$ we can write that
\begin{eqnarray}
    &&\bar X^{RE}_n(t_i)=\eta+\frac{T}{n}\sum\limits_{j=0}^{i-1}\tilde a(\xi_j,\bar X^{RE}_n(t_j))+\ \sum\limits_{j=0}^{i-1}\tilde b(t_j,\bar X^{RE}_n(t_j))\cdot\Delta W_j\notag\\
    &&\quad\quad\quad\quad\quad\quad+\delta_3\sum\limits_{j=0}^{i-1}\tilde b(t_j,\bar X^{RE}_n(t_j))\cdot\Delta Z_j,
\end{eqnarray}
and we have that
\begin{equation}
    \|\bar X_n^{RE}(t_i)\|\leq K+\sum\limits_{j=0}^{i-1}\|\tilde A_j\|+\Bigl\|\sum\limits_{j=0}^{i-1}\tilde B_j\Bigl\|+\sum\limits_{j=0}^{i-1}\|\tilde C_j\|,
\end{equation}
where
\begin{eqnarray}
    &&\tilde A_j=\frac{T}{n}\tilde a(\xi_j,\bar X^{RE}_n(t_j)),\\
    &&\tilde B_j = \tilde b(t_j,\bar X^{RE}_n(t_j))\cdot\Delta W_j,\\
    &&\tilde C_j = \delta_3\cdot\tilde b(t_j,\bar X^{RE}_n(t_j))\cdot\Delta Z_j.
\end{eqnarray}
Hence, for all $k=0,1,\ldots, n$
\begin{eqnarray}
\label{lem2_est1}
    &&\mathbb{E}\Bigl(\max\limits_{0\leq i\leq k}\|\bar X_n^{RE}(t_i)\|^r\Bigr)\leq c_rK^r+c_r\mathbb{E}\Bigl(\sum\limits_{j=0}^{k-1}\|\tilde A_j\|\Bigr)^r\notag\\
    &&+c_r\mathbb{E}\Bigl[\max\limits_{1\leq i \leq k}\Bigl\|\sum\limits_{j=0}^{i-1}\tilde B_j\Bigl\|^r\Bigr]+c_r\mathbb{E}\Bigl(\sum\limits_{j=0}^{k-1}\|\tilde C_j\|\Bigr)^r.
\end{eqnarray}
By Jensen inequality we have that
\begin{equation}
\label{lem2_est2}
    \mathbb{E}\Bigl(\sum\limits_{j=0}^{k-1}\|\tilde A_j\|\Bigr)^r\leq C_1+\frac{C_1}{n}\sum\limits_{j=0}^{k-1}\mathbb{E}\|\bar X^{RE}_n(t_j)\|^r.
\end{equation}
From Burkholder and Jensen inequality we obtain that
\begin{equation}
\label{lem2_est3}
    \mathbb{E}\Bigl[\max\limits_{1\leq i \leq k}\Bigl\|\sum\limits_{j=0}^{i-1}\tilde B_j\Bigl\|^r\Bigr]\leq C_2+\frac{C_2}{n}\sum\limits_{j=0}^{k-1}\mathbb{E}\|\bar X^{RE}_n(t_j)\|^r.
\end{equation}
Finally, since $\bar X^{RE}_n(t_j)$ and $\Delta W_j$ are independent, and
\begin{equation}
    \| \|\Delta W_j\|^{\beta}\|_r\leq c (T/n)^{\beta/2},
\end{equation}
we get that
\begin{eqnarray}
    &&\mathbb{E}\|\tilde C_j\|^r\leq \bar K_1\delta_3^r\mathbb{E}\Bigl[(1+\|\bar X^{RE}_n(t_j)\|)^r\cdot \Bigl((T/n)^\alpha+\|\Delta W_j\|^{\beta}\Bigr)^r\Bigr]\notag\\
    &&\leq \bar K_2 \delta_3^r \Bigl(1+\mathbb{E}\|\bar X^{RE}_n(t_j)\|^r\Bigr)\cdot \Bigl\|(T/n)^\alpha+\|\Delta W_j\|^{\beta}\Bigl\|_r^r\notag\\
    &&\leq \bar K_3\delta_3^r n^{-r\gamma}+\bar K_4\delta^rn^{-r\gamma}\mathbb{E}\|\bar X^{RE}_n(t_j)\|^r,
\end{eqnarray}
and hence
\begin{eqnarray}
\label{lem2_est4}
        &&\mathbb{E}\Bigl(\sum\limits_{j=0}^{k-1}\|\tilde C_j\|\Bigr)^r\leq n^{r-1}\sum\limits_{j=0}^{k-1}\mathbb{E}\|\tilde C_j\|^r\leq C_3\delta_3^r n^{r(1-\gamma)}\notag\\
        &&+C_4\delta_3^r n^{r(1-\gamma)-1}\sum\limits_{j=0}^{k-1}\mathbb{E}\|\bar X^{RE}_n(t_j)\|^r.
\end{eqnarray}
Combining \eqref{lem2_est1}, \eqref{lem2_est2}, \eqref{lem2_est3}, \eqref{lem2_est4} we arrive at
\begin{eqnarray}
     &&\mathbb{E}\Bigl(\max\limits_{0\leq i\leq k}\|\bar X_n^{RE}(t_i)\|^r\Bigr)\leq C_5(1+\delta_3^rn^{r(1-\gamma)})\notag\\
     &&+C_6(\delta_3^rn^{r(1-\gamma)-1}+n^{-1})\sum\limits_{j=1}^{k-1}\mathbb{E}\Bigl(\max\limits_{0\leq i\leq j}\|\bar X_n^{RE}(t_i)\|^r\Bigr).
\end{eqnarray}
By the discrete version of the Gronwall's lemma we get the thesis.
\end{proof}

\bf Acknowledgements\rm \\
This research was realized as a part of joint research project between AGH UST and NVIDIA.

\end{document}